\documentclass[11pt]{amsart}
\usepackage{amsmath,amsthm,amsfonts,amssymb,latexsym,enumerate,url,hyperref,color} 
\usepackage[shortlabels]{enumitem}

\usepackage[top=30mm,right=30mm,bottom=30mm,left=30mm]{geometry}

\usepackage{tikz-cd}
\usetikzlibrary{cd}

\theoremstyle{plain}

\newtheorem{thm}{Theorem}[section]
\newtheorem{lem}[thm]{Lemma}

\newtheorem{prop}[thm]{Proposition}

\newtheorem{rem}[thm]{Remark}

\newtheorem*{thmA}{THEOREM A}
\newtheorem*{thmB}{THEOREM B}
\newtheorem*{thmC}{THEOREM C}

\theoremstyle{definition}

\numberwithin{equation}{section}

\newcommand{\Syl}{\operatorname{Syl}}

\newcommand{\Stab}{\operatorname{Stab}}

\newcommand{\Aut}{\operatorname{Aut}}

\newcommand{\Irr}{\operatorname{Irr}}
\newcommand{\SL}{\operatorname{SL}}
\newcommand{\FF}{\mathbb{F}}

\newcommand{\OC}{\mathcal{O}}

\newcommand{\GL}{\operatorname{GL}}
\newcommand{\GU}{\operatorname{GU}}
\newcommand{\GO}{\operatorname{GO}}
\newcommand{\SO}{\operatorname{SO}}
\newcommand{\CO}{\operatorname{CO}}
\newcommand{\Sp}{\operatorname{Sp}}
\newcommand{\PGL}{\operatorname{PGL}}
\newcommand{\PSL}{\operatorname{PSL}}
\newcommand{\SU}{\operatorname{SU}}
\newcommand{\PSU}{\operatorname{PSU}}
\newcommand{\PSp}{\operatorname{PSp}}
\newcommand{\PO}{\mathrm{P}\Omega}

\def\Q{{\mathbb Q}}
\def\irr#1{{\rm Irr}(#1)}

\def\cent#1#2{{\bf C}_{#1}(#2)}
\def\syl#1#2{{\rm Syl}_#1(#2)}
\def\nor{\unlhd}
\def\oh#1#2{{\bf O}_{#1}(#2)}

\def\det#1{{\rm det}(#1)}
\def\ker#1{{\rm ker}(#1)}
\def\norm#1#2{{\bf N}_{#1}(#2)}

\newcommand{\diag}{{\mathrm {diag}}}

\newcommand{\St}{{\sf {St}}}

\newcommand{\AAA}{{\sf A}}
\newcommand{\SSS}{{\sf S}}
\newcommand{\CC}{{\mathbb C}}
\newcommand{\CB}{{\mathbf C}}
\newcommand{\RR}{{\mathbb R}}
\newcommand{\QQ}{{\mathbb Q}}
\newcommand{\ZZ}{{\mathbb Z}}

\newcommand{\ZB}{{\mathbf Z}}

\newcommand{\GC}{{\mathcal G}}

\newcommand{\eps}{\epsilon}
\newcommand{\varep}{\varepsilon}
\newcommand{\al}{\alpha}

\newcommand{\tw}[1]{{}^#1\!}

\newcommand{\ppd}{\mathsf{ppd}}

\newcommand{\pb}{\boldsymbol{\beta}}
\newcommand{\pc}{\boldsymbol{\gamma}}
\newcommand{\pd}{\boldsymbol{\delta}}
\newcommand{\pphi}{\boldsymbol{\varphi}}
\newcommand{\ppsi}{\boldsymbol{\psi}}

\def\irr#1{{\rm Irr}(#1)}
\def\cent#1#2{{\bf C}_{#1}(#2)}

\def\syl#1#2{{\rm Syl}_#1(#2)}

\def\nor{\trianglelefteq\,}

\def\irr#1{{\rm Irr}(#1)}

\def\cent#1#2{{\bf C}_{#1}(#2)}

\def\syl#1#2{{\rm Syl}_#1(#2)}
\def\nor{\trianglelefteq\,}
\def\norm#1#2{{\bf N}_{#1}(#2)}
\def\oh#1#2{{\bf O}_{#1}(#2)}

\begin{document}

\renewcommand{\thefootnote}{\fnsymbol{footnote}}
\footnotesep6.5pt

\title[$p$-rational characters and Sylow $p$-subgroups]
{Degrees of $p$-rational characters and normality of Sylow $p$-subgroups}

\author[S. Dolfi]{Silvio Dolfi}
\address{Dipartimento di Matematica, Universit\`a di Firenze, 50134 Firenze, Italy}
\email{silvio.dolfi@unifi.it}

\author[P. H. Tiep]{Pham Huu Tiep}
\address{Department of Mathematics, Rutgers University, Piscataway, NJ 08854, USA}
\email{tiep@math.rutgers.edu}

\author[Yu Zeng]{Yu Zeng}
\address{Department of Mathematics, Suzhou University of Technology, No.99 South Third Ring
				   Road,
				   Changshu, Jiangsu, 215500, China}
\email{yuzeng2004@163.com}

\thanks{The first author gratefully acknowledges the support of INDAM-GNSAGA.
  The second author gratefully acknowledges the support of the NSF (grant
 DMS-2200850) and the Joshua Barlaz Chair in Mathematics.}
\thanks{It is a pleasure to thank Gabriel Navarro for helpful discussions.}



\begin{abstract}
Several refinements of (the normality part of) the celebrated It\^o--Michler theorem were obtained during the last two decades, in which the condition of having 
$p'$-degree, for a fixed prime $p$, is imposed only on some subsets of complex irreducible characters of a finite group $G$. We prove further extensions of these
results, where this condition is now imposed on the irreducible characters which lie above the principal character of a Sylow $p$-subgroup and are either $p$-rational,
or strongly real when $p=2$.
\end{abstract}

\maketitle

\tableofcontents
 
\section{Introduction}
During the last two decades, several refinements of (the normality part of) the celebrated It\^o--Michler theorem were obtained, which characterize the normality of 
Sylow $p$-subgroups of a finite group $G$ for a fixed prime $p$, by imposing the condition of having 
$p'$-degree only on some subsets of complex irreducible characters of $G$.
In~\cite{MN} G. Malle and G. Navarro prove that the normality of Sylow $p$-subgroups $P$ of $G$ can be characterized by imposing this condition on the irreducible characters of $G$
lying over the principal character of $P$. 
In~\cite{NT}, a similar criterion is found using the degrees of the $p$-rational characters of $G$. When $p=2$, a criterion using the degrees of
real-valued characters is provided in \cite{DNT}. Also see \cite[Theorem C]{MN2} for another criterion of different kind. 

The following main theorems of this paper extend both the main results in~\cite{MN} and~\cite{NT}, respectively in \cite{MN} and \cite{DNT}.

\begin{thmA}
  	Let $G$ be a finite group, $p$ a prime number and let $P$ be a Sylow $p$-subgroup of $G$.
	Then the following conditions are equivalent.
	\begin{description}
		\item[(1)] Every $p$-rational irreducible constituent $\chi$ of $(1_P)^G$ has $p'$-degree.
		\item[(2)] $\chi(x)\neq 0$ for every $x\in P$ and every $p$-rational irreducible constituent $\chi$ of $(1_P)^G$.
		\item[(3)] $P$ is normal in $G$.
	\end{description}
      \end{thmA}

\begin{thmB}
  	Let $G$ be a finite group and let $P$ be a Sylow $2$-subgroup of $G$.
	Then the following conditions are equivalent.
	\begin{description}
		\item[(1)] Every strongly real  irreducible constituent $\chi$ of $(1_P)^G$ has odd degree.
		\item[(2)] $\chi(x)\neq 0$ for every $x\in P$ and every strongly real irreducible constituent $\chi$ of $(1_P)^G$.
		\item[(3)] $P$ is normal in $G$.
	\end{description}
      \end{thmB}
      
Part of Theorem A (namely the equivalence of the conditions {\bf A(1)} and {\bf A(2)}), and 
part of Theorem B (namely the equivalence of the conditions {\bf B(1)} and {\bf B(2)}, with real characters instead of strongly real characters) were independently 
conjectured by Navarro, see \cite[Conjecture 5.8]{N2}. 
As one may expect, our proofs follow the same strategy of \cite{DNT}, \cite{MN}, \cite{NT}, reducing the statements to 
a condition on almost simple groups, see Theorems \ref{simple1} and \ref{simple2}. The verification of the latter condition is however more involved, and 
relies heavily on the Deligne--Lusztig theory.

One may ask whether one can have some strengthenings of Theorem A, respectively Theorem B, where in {\bf A(1)} or {\bf A(2)}, respectively
in {\bf B(1)} or {\bf B(2)}, the condition is imposed only on the constituents $\chi$ of $(1_P)^G$ {\it that occur with multiplicity coprime to $p$} -- such a strengthening 
of the main result of \cite{MN} was obtained in \cite{GLLV}. As we show in 
Remark \ref{p23}, both of these strengthenings of Theorem A (at least for $p=3$)  and Theorem B are false. However, we can prove the following result.

\begin{thmC}
  	Let $G$ be a finite group and let $P$ be a Sylow $2$-subgroup of $G$.
	Then the following conditions are equivalent.
	\begin{description}
		\item[(1)] Every $2$-rational character $\chi \in \Irr(G)$ with $2 \nmid [\chi_P,1_P]$ has odd degree.
		\item[(2)] $\chi(x)\neq 0$ for every $x\in P$ and every $2$-rational  $\chi \in \Irr(G)$ with $2 \nmid [\chi_P,1_P]$.
		\item[(3)] $P$ is normal in $G$.
	\end{description}
      \end{thmC}

\section{Preliminaries}

By Frobenius reciprocity, a character $\chi$ is an irreducible constituent  of the induced character
$(1_P)^G$, where $P$ is a subgroup of the group $G$, if and only if $1_P$ is an irreducible constituent of
the restriction $\chi_P$; in this case, we will say that $\chi$ \emph{lies over} $1_P$. 

\begin{lem}\label{Ldef0}
	Let $p$ be a prime number and let $P$ be a Sylow $p$-subgroup of a finite group $G$.
	If $\chi \in \irr G$ has $p$-defect zero, then $\chi$ is a $p$-rational character that contains
        $1_P$ with multiplicity coprime to $p$  and such that $\chi(x) = 0$ for every non-trivial element $x \in P$.
      \end{lem}
      
\begin{proof}
  By \cite[Theorem 8.17]{Is}, $\chi(x)=0$ for every element $x$ in $G$ of order divisible by $p$.
  Hence, $\chi$ is $p$-rational and $[\chi_P,1_P]= \chi(1)/|P|$, a positive integer coprime to $p$.
\end{proof}

\begin{lem}\label{Lprat}
  Let $G$ be a finite group, $N$ a normal subgroup of $G$ and $p$ a prime number.
  If $\theta \in \irr N$ has $p$-defect zero and $|G/N|$ is coprime to $p$, then
  there exists a $p$-rational character $\chi \in \irr G$ such that $\chi$ lies over $\theta$. 
\end{lem}

\begin{proof}
  This is a particular case of~\cite[Corollary 2.4]{NT}.
\end{proof}

\begin{lem}\label{Lgab}
 Suppose that $G$ is a finite group with a normal subgroup $N$ such that $G/N$ is a $p$-group, $P \in \syl pG$, $Q=P\cap N$.
 Let $\theta \in \irr N$ be $G$-invariant such that $m:=[\theta_Q, 1_Q]$ is not divisible by $p$.
 \begin{enumerate}[\rm(i)]
 \item Then $\theta$ has a unique extension $\hat\theta \in \irr G$ such that if $\hat\theta_P=\Psi + \Delta$,
 where $\Psi$ is a character of $P/Q$, $\Delta$ is a character of $P$ or zero with $[\Delta_Q, 1_Q]=0$, then ${\rm det}(\Psi)=1_P$. 
 In particular, $\QQ(\hat\theta)=\QQ(\theta)$.
 \item Suppose that $m=1$.
 Then $[\hat\theta_P, 1_P]$ is not divisible by $p$. 
 \item Suppose that $p=2$, and either $\theta$ is $p$-rational or real-valued.
 Then $\theta$ has an extension $\tilde\theta$ to $G$ such that $p \nmid [\tilde\theta_P, 1_P]$ and $\QQ(\tilde\theta)=\QQ(\theta)$. 
 Moreover, if $\theta$ is real-valued, then $\tilde{\theta}$ is strongly real.
 \end{enumerate}
 \end{lem}
 
\begin{proof}
(i) is \cite[Corollary 6.10]{N}.

\smallskip
(ii) Here we have $\Psi_Q= 1_Q$, which implies $\Psi=\det\Psi=1_P$. As $[\Delta_Q,1_Q]=0$, 
it follows that $[\hat\theta_P,1_P]=1$. 

\smallskip
(iii) If $\theta$ is $2$-rational, let $\Gamma:=\mathrm{Gal}\bigl(\QQ(\exp\frac{2\pi i}{|G|})/\QQ(\exp\frac{2\pi i}{|G|_{2'}})\bigr)$. If $\theta$ is real-valued,
let $\Gamma:=\langle \tau \rangle$, where $\tau$ is the complex conjugation automorphism of $\QQ(\exp\frac{2\pi i}{|G|})$. By assumption,  
$\theta$ is $\Gamma$-fixed, and so is $\hat\theta$. It follows that the $P/Q$-character $\Psi$, which has odd degree $m$, is also $\Gamma$-fixed.
Now write 
$$\Psi = \sum_\OC a_\OC \bigl(\sum_{\lambda \in \OC}\lambda\bigr),$$
where $\OC$ runs over the set of $\Gamma$-orbits on the set of irreducible constituents of $\Psi$ and $a_\OC \in \ZZ_{\geq 1}$. 
As $\Gamma$ is a finite $2$-group, the length of any $\Gamma$-orbit $\OC$ is a $2$-power.
Since $m=\Psi(1)$ is odd, there must exist a $\Gamma$-orbit $\OC_1 = \langle \gamma \rangle$ of length $1$, such that 
$\gamma(1)$ and $a_{\OC_1}$ are odd. But $P/Q$ is a $2$-group, so $\gamma(1)=1$. As $\gamma$ is $\Gamma$-fixed, it follows that 
$\gamma$ has order $1$ or $2$; in particular, $\gamma$ is rational. Also note that $[\hat\theta_P,\gamma]=a_{\OC_1}$ is odd.
Now we can view $\gamma$ as a linear character of $G/N$ by identifying $G/N$ with $P/Q$, and take 
$\tilde\theta:=\gamma\hat\theta$. 
Finally, if $\theta$ is real-valued, then $\tilde{\theta}$ is real-valued and $[\tilde{\theta},(1_P)^G]$ is odd.
Hence, as the permutation module affording the character $(1_P)^G$ admits a nondegenerate $G$-invariant symmetric bilinear form,  
by \cite[Lemma 2.4(iii)]{T} the character $\tilde{\theta}$ is strongly real.
\end{proof}

Note that Lemma \ref{Lgab}(iii) may fail when $p>2$. In particular, it fails when $p=2^n+1 \geq 3$ is a Fermat prime. For such a prime $p$, we can take 
$G = 2^{1+2n}_- \rtimes \mathsf{C}_p$ and note that the unique faithful irreducible character $\theta$ of $N$ (of degree $2^n$) 
has a unique $p$-rational extension to $G$ which however does not lie over $1_P$.

\begin{lem}\label{triv}
Let $G$ be a finite group, $\chi \in \Irr(G)$, and let $P$ be a $p$-subgroup of $G$.
\begin{enumerate}[\rm(i)]
\item If $\chi$ is real-valued and of $2$-defect zero, then $\chi$ is strongly real.
\item Suppose there is a prime $\ell \neq p$ such that the restriction $\chi^\circ$ of $\chi$ to
$\ell'$-elements of $G$, considered as an $\ell$-Brauer character, contains $1_G$ as an irreducible constituent.
Then $\chi_P$ contains $1_P$. 
\end{enumerate}
\end{lem}

\begin{proof}
(i) This is Remark 2 on \cite[p. 254]{R}. 

\smallskip
(ii) By assumption, $\chi_P = (\chi^\circ)_P$ contains $1_P$. In fact the statement is true for $P$ any $\ell'$-subgroup.
\end{proof}

\begin{lem}\label{ext}
Let $G$ be a finite group with a normal subgroup $N$, $\alpha \in \Irr(N)$ be $G$-invariant, and let $\beta$ be a real-valued 
$G$-invariant irreducible $\ell$-Brauer character of $N$ for some prime $\ell$. Also assume $P \in \Syl_p(G)$ for some prime 
$p \neq \ell$.
\begin{enumerate}[\rm(i)]
\item Let $\rho$ be a character of $G$
not necessarily irreducible with $[\al,\rho_{N}] = 1$ and  $\QQ(\rho) \subseteq \QQ(\alpha)$.
Then $\rho$ contains an irreducible constituent $\chi \in \Irr(G)$ such that $\QQ(\chi) = \QQ(\alpha)$ and $\chi_{N} = \al$.
\item Suppose there is a subgroup $H$ such that $[\al,\rho_N]=1$ for $\rho:=(1_H)^G$ and moreover $\al$ is real-valued.
Then $\rho$ contains an irreducible constituent $\chi \in \Irr(G)$ such that $\chi$ is strongly real and $\chi_{N} = \al$.
\item  Suppose $\chi$ is a real-valued character of $G$ not necessarily irreducible  such that
$(\chi^\circ)_N$ contains $\beta$ with multiplicity one, for the restriction $\chi^\circ$ of $\chi$ to 
$\ell'$-elements of $G$. Then $\chi^\circ$ contains a real-valued irreducible constituent $\gamma$ with $\gamma_N=\beta$.
Furthermore, if $\beta=1_N$, or $p=2=|G/N|$ and $\chi_Q$ contains $1_Q$ for $Q = P \cap N \in \Syl_2(N)$, 
then there exists a linear character $\lambda$ of $G/N$ of order $\leq \gcd(2,p)$ such that $(\chi\lambda)_P$ contains $1_P$.
\end{enumerate}
\end{lem}

\begin{proof}
(i) This is \cite[Lemma 5.1]{NT1}.

\smallskip
(ii) Clearly, $\rho_N$ is afforded by a $\RR$-representation, hence the Schur index of $\al$ over $\RR$ divides $[\al,\rho_N]=1$ by
\cite[Corollary (10.2)(c)]{Is}. Since $\al$ is real-valued, it follows from \cite[Corollary (10.2)(e)]{Is} that $\al$ is strongly real. By 
(i), $\rho$ contains a real-valued constituent $\chi \in \Irr(G)$ such that $\chi_{N} = \al$. As $\al$ is strongly real, so is $\chi$.

\smallskip
(iii) Let $\gamma$ denote an irreducible constituent of $\chi^\circ$ lying above $\beta$. Since $(\chi^\circ)_N$ contains $\beta$ with multiplicity one,
$\gamma_N=\beta$. Since $\chi$ is real-valued, $\bar\gamma$ is an irreducible constituent of $\chi^\circ$ lying above $\bar\beta=\beta$, and so we 
deduce that $\gamma=\bar\gamma$. 

Now assume that $\beta=1_N$. Then $\gamma$ is a real-valued $\ell$-Brauer character of degree $1$ of $G/N$,
whence either $\gamma_{PN}=1_{PN}$, or $p = 2$ and $\gamma_{PN}$ is a linear character of order $2$ of the $p$-group $PN/N \cong P/Q$. 
In the former case $\chi_P$ contains $1_P$ by Lemma \ref{triv}(ii). In the latter case  $(\chi\gamma)_P$ contains
$1_P$ again by Lemma \ref{triv}(ii).

Next assume that $p=2=|G/N|$ and $\chi_Q$ contains $1_Q$ for $Q = P \cap N \in \Syl_2(N)$. Consider a $\CC G$-module $V$
affording $\chi$. Then $U \neq 0$ for the subspace $U$ of $Q$-fixed points on $V$. Now the group $P/Q \cong G/N$ of order $2$ acts on $U$, so $U$ admits a $P$-subspace that affords a linear character $\lambda$ of $P/Q$ of order $\leq 2$. Viewing $\lambda$ as a character of 
$G/N$, we conclude that $(\chi\lambda)_P$ contains $1_P$.
\end{proof}

\begin{lem}\label{index2}
Let $N$ be a normal subgroup of index $2$ of a finite group $G$, and let $\theta \in \Irr(N)$ be $G$-invariant.
\begin{enumerate}[\rm(i)]
\item Suppose there is a complex character $\chi$ of $G$ not necessarily irreducible such that $m:=[\chi_N,\theta]$ is odd and 
$\QQ(\chi) \subseteq \QQ(\theta)$.
Then any extension $\tilde\theta$ of $\theta$ to $G$ satisfies $\QQ(\tilde\theta) = \QQ(\theta)$.
\item Assume in addition that $\theta$ is rational-valued and the character $\chi$ in {\rm (i)} is afforded by a $\QQ G$-module. Then 
$\tilde\theta$ is afforded by a $\QQ G$-module. 
\end{enumerate}
\end{lem}

\begin{proof}
(i) Since $G/N$ is cyclic of order $2$, $\theta$ extends to a character $\tilde\theta$ of $G$, and the other extension is $\tilde\theta\lambda$,
where $\lambda \in \Irr(G/N)$ has order $2$; in particular $\QQ(\tilde\theta\lambda)=\QQ(\tilde\theta)$.
Now consider the $\theta$-isotypic component $\nu$ of $\chi$, so that
$\nu_N=m\theta$, and we can write $\nu=a\tilde\theta+b\tilde\theta\lambda$ with $a,b \in \ZZ_{\geq 0}$ and $a+b=m$.
Since $\QQ(\chi) \subseteq \QQ(\theta)$, any $\sigma \in \mathrm{Gal}(\QQ(\tilde\theta)/\QQ(\theta))$ preserves $\nu$, 
and hence $\QQ(\nu) \subseteq \QQ(\theta)$. Clearly, if $g \in N$ then $\tilde\theta(g)=\theta(g) \in \QQ(\theta)$.
Suppose $g \notin N$. Then $\lambda(g)=-1$, and so
$$\QQ(\theta) \ni \nu(g) = a\tilde\theta(g)+b\tilde\theta\lambda(g) = (a-b)\tilde\theta(g).$$
As $2 \nmid m=a+b$, $a-b$ is a nonzero integer, whence $\tilde\theta(g) \in \QQ(\theta)$.

\smallskip
(ii) By (i), $\tilde\theta$ is rational-valued, whence its Schur index $m_\QQ(\tilde\theta) \leq 2$ by the Brauer--Speiser theorem \cite[p. 171]{Is}.
In the notation of (i), we may assume $a=[\chi,\tilde\theta]$ is odd. Now $m_\QQ(\tilde\theta)|a$ by \cite[Corollary (10.2)(c)]{Is}, 
so $m_\QQ(\tilde\theta)=1$.
\end{proof}


\section{Almost simple groups. I}

The goal of this section and the next section are devoted to the proofs of the following two results.

\begin{thm}\label{simple1}
 Let $S$ be a finite non-abelian simple group and $p$ a prime divisor of $|S|$. Then at least one of the following two statements holds.
 \begin{enumerate}[\rm(i)] 
 \item $S$ admits an irreducible character $\alpha$ and a $p$-element $x \in S$ such that 
$\alpha^\varsigma(x)=0$ for all $\varsigma \in \Aut(S)$ and $\alpha$ extends to 
a $p$-rational character $\beta$ of $I=I_{\Aut(S)}(\alpha)$ which lies above $1_P$, where $P$ is a Sylow $p$-subgroup of $I$. Moreover, if 
$p=2$ then $\beta$ can be chosen so that $2 \nmid [\beta_P,1_P]$.
\item Let $A$ be an almost-simple group with socle $S$ and assume that $|A/S|$ is a power of $p$.
  Then there exists a character $\alpha \in \irr S$ of $p$-defect zero  such that 
   $\alpha$ extends to a $p$-rational character $\beta$ of  $I = I_A(\alpha)$ and $\beta$ lies over $1_P$,  where $P$ is a Sylow $p$-subgroup of $I$. 
Moreover, if  $p=2$ then $\beta$ can be chosen so that $2 \nmid [\beta_P,1_P]$.
\end{enumerate}   
 \end{thm}
 
\begin{thm}\label{simple2}
Let $S$ be a finite non-abelian simple group. Then at least one of the following two statements holds.
  
\begin{enumerate}[\rm(i)]  
\item $S$ admits an irreducible character $\alpha$ and a $2$-element $x \in S$ such that 
$\alpha^\varsigma(x)=0$ for all $\varsigma \in \Aut(S)$ and $\alpha$ extends to 
a strongly real character of $I=I_{\Aut(S)}(\alpha)$ which lies above $1_P$, where $P$ is a Sylow $2$-subgroup of $I$. 

\item If A is an almost simple group with socle $S$ and $A/S$ is a $2$-group, then there exists a character
$\alpha \in \irr S$ of $2$-defect zero  such that 
$\alpha$ extends to a strongly real character $\beta$ of  $I = I_A(\alpha)$ and $\beta$ lies over $1_P$,  where $P$ is a Sylow $2$-subgroup of $I$. 
\end{enumerate} 
\end{thm}
 
We first deal with certain small almost simple groups. 

\begin{lem}\label{small}
 \begin{itemize}
\item[\rm(a)]  Theorem {\rm\ref{simple1}} holds if $S$ is isomorphic to either a sporadic simple group or to at least one of the following simple groups
$\AAA_5$, $\AAA_6$, $\PSp_4(5)$, $\PSp_4(7)$, $\PSp_8(3)$, $\Omega_9(3)$, $\tw3 D_4(2)$, $\tw2 F_4(2)'$, $G_2(3)$, $G_2(4)$.
\item[\rm(b)]  Theorem {\rm\ref{simple2}} holds if $S$ is isomorphic to either a sporadic simple group or to  one of the following simple groups
$\AAA_5$, $\AAA_6$, $\tw2 F_4(2)'$.
  \end{itemize}
\end{lem}

\begin{proof} 
We first assume that  $S$ is isomorphic to $\tw3 D_4(2)$. Then $|\mathrm{Out}(S)| = 3$ and, since $S$ has $p$-defect zero irreducible characters for every prime $p$, we can clearly reduce to the case $p=3$. As $\mathrm{Aut}(S)$ has an irreducible character $\gamma$ with degree $6318$, hence of $3$-defect zero,
then any irreducible constituent $\alpha$ of the restriction $\gamma_S$ is a $3$-defect zero character of $S$
and $I_{\mathrm{Aut}(S)}(\alpha) = S$ by~\cite[Corollary 6.20]{Is}. So part (ii) of Theorem~\ref{simple1} holds for
$\tw3 D_4(2)$.

For all the simple groups $S$ listed in part (a), with the only exception of $\tw3 D_4(2)$, $\mathrm{Out}(S)$ is a $2$-group.
Hence, if $S$ has $p$-defect zero irreducible characters for a given prime $p$, then  part (ii) of Theorem~\ref{simple1} follows immediately for $p$ odd and it is a consequence of Lemma~\ref{Ldef0} and Lemma~\ref{Lgab}(iii) for $p=2$.

Since the listed simple groups possess irreducible characters of $\ell$-defect zero for all primes $\ell \geq 5$,
this leaves us with only the cases $p = 2$ and $p=3$ to check. 
For $p=3$, we recall that if $S$ is a sporadic simple group with no
irreducible character of $3$-defect zero, then $S \in \{Suz, Co_3 \}$. 
Let $P$ be a Sylow $3$-subgroup of $S$.
One checks by~\cite{GAP} that  in both cases there exists a rational character $\alpha \in \mathrm{Irr}(S)$
lying above $1_P$
(with $\alpha(1) = 780$  if $S = Suz$ and $\alpha(1) = 5544$ if $S = Co_3$) such that $\alpha(x) = 0$ for all
$x \in S$ with $|x| =9$.
Hence, part (i) of Theorem~\ref{simple1} holds for $Suz$ and $Co_3$.

For $p=2$, the sporadic simple groups with no irreducible character of $2$-defect zero are considered below. 

In order to prove part (b) and to complete the proof of part (a) for the remaining sporadic simple groups and $p=2$,
we observe that if $S$ has a \emph{real} $2$-defect zero irreducible character,
then part (ii) of Theorem~\ref{simple2} holds for $S$ by Lemma~\ref{Lgab}(iii).
Using~\cite{GAP} or \cite{atlas} one checks that among the  simple groups listed in  part (b),
only the following groups fail to have a real $2$-defect zero irreducible character:
\begin{equation}
  \label{eq:1}
  M_{11},  M_{12} ,  M_{22} ,  M_{23} ,  M_{24} ,  HS ,  J_2 ,  McL ,  Ru ,  Suz ,  Co_1 ,  Co_3 ,  B \; . 
\end{equation}

For all the  groups $S$ in~(\ref{eq:1}), with  $S \not\cong M_{11}$, one checks by~\cite{GAP} that there exists a
\emph{rational}  character $\alpha \in \mathrm{Irr}(S)$ such that $[\alpha, 1_P^S]$ is odd, where $P$ is a Sylow $2$-subgroup of $S$, and a $2$-element $x\in S$ such that for every element $y \in S$ with $|y| = |x|$ it is
$\alpha(y)=0$ (so, in particular, $\alpha(1)$ is even).
Instances of  triples $(S, \alpha(1), |x|)$ are as follows:
$( M_{12}, 144, 4)$,  $( M_{22}, 210, 8)$, $(  M_{23}, 22, 8)$, $(M_{24}, 2024, 8)$,   $(HS, 22, 8)$,
$(J_2, 288,4)$, $(McL, 22,8)$, $(Ru,3654,16)$,  $(Suz,3432,8)$,  $(Co_1, 376740, 16)$,  $(Co_3,7084,8)$,
$(B, 347643114,32)$. 
Hence, part (i) of both Theorem~\ref{simple1} and  Theorem~\ref{simple2} follows by an application of Lemma~\ref{Lgab}(iii).

Finally, if $S \cong M_{11}$, then $S = \mathrm{Aut}(S)$ and $(1_P)^S$, where $P$ is a Sylow $2$-subgroup of $S$, has  a rational irreducible constituent $\alpha \in \mathrm{Irr}(S)$ with $\alpha(1) = 44$ and
$\alpha(x) = 0$ for all $x \in S$ such that $|x| =4$. Hence, part (i) of Theorem~\ref{simple2} holds for $M_{11}$. ($M_{11}$ has $2$-defect zero irreducible characters, so they are $2$-rational characters by Lemma~\ref{Ldef0}, and $\mathrm{Out}(M_{11}) = 1$, so part (ii) of Theorem~\ref{simple1} certainly holds for $M_{11}$.)
\end{proof}

\begin{rem}\label{p23}
{\em
\begin{enumerate}[\rm(i)]
\item Note that $M_{11}$ provides an example of a non-abelian simple group where no real irreducible constituent of $(1_P)^G$ contains $1_P$ with odd multiplicity,
for $P \in \Syl_2(G)$. Thus the ``odd-multiplicity'' strengthening of Theorem B is false.
\item Consider $G=\PGL_3(4)$ and $P \in \Syl_3(G)$. Using \cite{GAP} one can check that if $\chi$ is a $3$-rational irreducible constituent of 
$(1_P)^G$ with multiplicity coprime to $3$ then $\chi$ has degree $1$, $20$, or $64$ (and so non-vanishing at any $3$-element of $G$), but yet $P$ is not normal in $G$.  Thus the ``coprime to $p$ multiplicity'' strengthening of Theorem A is false when $p=3$.
\end{enumerate}
}
\end{rem}

\begin{prop}\label{almost-simple-others}
 Let $S$ be an alternating group $\AAA_n$, $n \geq 7$. 
 Then both Theorem {\rm\ref{simple1}} and Theorem {\rm\ref{simple2}} hold for $S$.
\end{prop}

\begin{proof}
  For  $S=\AAA_n$, with $n \geq 7$, and $A = \mathrm{Aut}(S) \cong \SSS_n$, we have $|A/S| = 2$.
  Thus, for every $\alpha \in \irr S$ and $I = I_A(\alpha)$, $[I:S] \leq 2$  and hence
  by~\cite[Corollary 6.20]{Is} there exists a character $\beta \in \irr I$ such that $\beta_S = \alpha$.
  If $\alpha$ has $p$-defect zero and $p \neq 2$, then $\beta$ is a $p$-defect zero character of $I$ and hence $\beta$ is $p$-rational and it lies over $1_P$ by Lemma~\ref{Ldef0}. So, part (i) of Theorem~\ref{simple1} holds. 
  Therefore,  by \cite[Corollary 2]{GO}, we are reduced to the case  $p \in \{2,3\}$.
  For $p \in \{2,3\}$ and $P$ a Sylow $p$-subgroup of $S$, by~\cite[Theorem 3.17]{GLLV} there exist an $A$-invariant character $\alpha \in \mathrm{Irr}(S)$
  such that $p\nmid [\alpha_P, 1_P]$ and an element $x \in P$ such that $\alpha(x) = 0$.
  By the remark above, $\alpha$ is the restriction of a character of $A$, so $\alpha$ is rational and strongly real and hence part (i) of both Theorem~\ref{simple1} and Theorem~\ref{simple2} holds. 
\end{proof}

\begin{thm}\label{simple2a}
Theorem {\rm \ref{simple1}} holds if $p=2$.
\end{thm} 
 
\begin{proof}
We will assume that $S$ is not isomorphic to any of the simple groups listed in part (a) of Lemma \ref{small}.  Note that it suffices to show that $S$ admits an
irreducible character $\theta$ of $2$-defect zero. Indeed, such a $\theta$ is $2$-rational and contains $1_Q$ with odd multiplicity for $Q \in \Syl_2(S)$ 
by Lemma \ref{Ldef0}. 
Applying Lemma \ref{Lgab}(iii) to $(G,N):=(I_A(\theta),S)$, we see
that $\theta$ has a $2$-rational extension $\beta$ to $I=I_A(\theta)$ with $2 \nmid [\beta_P,1_P]$ for $P \in \Syl_2(I)$, as desired. 

The proof follows the same strategy as in \cite{Wi}, relying on the Deligne--Lusztig theory \cite{C, DM, L}, but it is more convenient for us to align it with the subsequent treatment of odd primes $p$. If $S$ is a simple group of Lie type in characteristic $2$, then the Steinberg character $\St$ of $S$ 
has $2$-defect zero, in fact of degree $|S|_2$, and it is rational, so we are done in this case by Lemma \ref{Lgab}(ii).

Assume now that $S$ is a simple group of Lie type in characteristic $r > 2$. 
We will view $S$ as $[G,G]$, where $G=\GC^F$ for some simple algebraic group $\GC$ of adjoint type defined over a field of characteristic $r > 2$
and a Steinberg endomorphism $F:\GC \to \GC$. Let the pair $(\GC^*,F^*)$ be dual to $(\GC,F)$ and let $L:=(\GC^*)^{F^*}$. We will find some suitable
semisimple $2'$-element $s \in L$ and consider the corresponding semisimple character $\chi_s$ of $G$. Since $|s|$ is odd, $\chi_s$ is $2$-rational.
Moreover, $\theta:=(\chi_s)_S$ is irreducible by \cite[Proposition 4.3(iii)]{MT1}. As $\theta(1)=[L:\CB_L(s)]_{r'}$ and $|L|=|G|$, it suffices to find $s$ 
such that 
\begin{equation}\label{ind2}
  |\CB_L(s)|_2=|G/S|_2.
\end{equation}
In the sequel, $q = r^f$ for some $f \in \ZZ_{\geq 1}$, and we will sometimes use the existence of primitive prime divisors $\ell = \ppd(a,m)$, i.e. 
a prime divisor $\ell$ that divides $a^m-1$ but not $\prod^{m-1}_{i=1}(a^i-1)$ for $m \in \ZZ_{\geq 2}$ and an integer $a \neq 0,\pm 1$, see \cite{Zs}.

\smallskip
(i) Suppose $S = \PSL_2(q)$ with $q \geq 7$. Then $L \cong \SL_2(q)$ and $|G/S|=2$. If $q \equiv 1 \pmod4$, then we choose $s \in L$ of odd order
$(q+1)/2$, which implies $|\CB_L(s)|=q+1$ and so \eqref{ind2} holds.  If $q \equiv 3 \pmod4$, then we choose $s \in L$ of odd order
$(q-1)/2$, which implies $|\CB_L(s)|=q-1$ and so \eqref{ind2} holds again.

Next suppose that $S = \PSL^\eps_n(q)$ with $n \geq 3$ and $\eps = \pm 1$, where $\PSL^\eps$ stands for $\PSL$ when $\eps = 1$ and for
$\PSU$ when $\eps = -1$. Then $L \cong \SL^\eps_n(q)$ and $|G/S|=\gcd(n,q-\eps)=:d$. If $n=3$, then we choose $s \in L$ of odd order $q^2+\eps q+1$, which implies $|\CB_L(s)|=q^2+\eps q+1$ and so \eqref{ind2} holds. Suppose $n \geq 4$. Note that 
$$\gcd\bigl(q^{n-1}-\eps^{n-1},\frac{q^n-\eps^n}{q-\eps}\bigr) = d,$$ 
so the $2$-part of at least one of $N_1:=q^{n-1}-\eps^{n-1}$ and $N_2:=(q^n-\eps^n)/(q-\eps)$ is equal to $|G/S|_2$. Now, for each $i = 1,2$, we can
choose $s_i \in L$ of odd order $\ppd(\eps q,n-1)$ for $i=1$ and $\ppd(\eps q,n)$ for $i=2$. Then $|\CB_L(s_i)|=N_i$, and so we are done.

Consider the case where $S = \PSp_{2n}(q)$ with $n \geq 2$ (so $L = \mathrm{Spin}_{2n+1}(q)$), or $S = \Omega_{2n+1}(q)$ with
$n \geq 3$ (so $L = \Sp_{2n}(q)$). Then $|G/S|=\gcd(2,q-1)=2 = \gcd(q^n+1,q^n-1)$. So exactly one of $q^n+1$ and $q^n-1$, denoting it 
$N$, is congruent to $2$ modulo $4$. Now we choose $s \in L$ of odd order $N/2$, for which we have $|\CB_L(s)|=N$ and so \eqref{ind2} holds. 

Suppose that $S = \PO^+_{2n}(q)$ with $n \geq 4$, so $L = \mathrm{Spin}^+_{2n}(q)$ and $|G/S|=\gcd(4,q^n-1)=:d$.  
If $d=2$, then we choose $s \in L$ of odd order $(q^n-1)/2$, for which we have $|\CB_L(s)|=q^n-1$ and so \eqref{ind2} holds. 
If $q \equiv 1 \pmod4$, then we choose $s \in L$ of odd order $(q^{n-1}+1)/2$, for which we have 
$|\CB_L(s)|=(q^{n-1}+1)(q+1)$, and hence \eqref{ind2} holds. In the remaining case we have $d=4$, $q \equiv 3 \pmod{4}$, and 
$2|n$. Now we choose $s \in L$ of odd order $(q^{n-1}-1)/2$, for which we have 
$|\CB_L(s)|=(q^{n-1}-1)(q-1)$, and \eqref{ind2} holds again.

Next assume that $S = \PO^-_{2n}(q)$ with $n \geq 4$, for which we have $L = \mathrm{Spin}^-_{2n}(q)$ and $|G/S|=\gcd(4,q^n+1)=:d$.  
If $d=2$, then we choose $s \in L$ of odd order $(q^n+1)/2$, for which we have $|\CB_L(s)|=q^n+1$ and so \eqref{ind2} holds. 
In the remaining case we have $d=4$, $q \equiv 3 \pmod{4}$, and 
$2 \nmid n$. Now we choose $s \in L$ of odd order $(q^{n-1}+1)/2$, for which we have 
$|\CB_L(s)|=(q^{n-1}+1)(q-1)$, and \eqref{ind2} holds again.

\smallskip
(ii) Now we handle the exceptional groups of Lie type. First suppose that $S = \tw2 G_2(q)$, where $r=3$ and $2 \nmid f \geq 3$. Then using 
\cite[Lemma 2.3]{MT}, we choose $s \in L \cong S$ of odd order $\ppd(r,6f)$, for which $|\CB_L(s)|$ is odd. 

Next assume that $S$ has type $E_7(q)$. Then $|G/S|=2=\gcd(q^7-1,q^7+1)$. So exactly one of $q^7+1$ and $q^7-1$, denoting it 
$N$, is congruent to $2$ modulo $4$. Now using \cite[Table 7]{JPRW}, we can choose $s \in L$ of odd order $N/2$, for which we have 
$|\CB_L(s)|=N$ and so \eqref{ind2} holds. 

For the remaining groups
$S$ of type $G_2(q)$, respectively $\tw3 D_4(q)$,  $F_4(q)$, $E_6(q)$, $\tw2 E_6(q)$, and $E_8(q)$, we can again use \cite[Lemma 2.3]{MT}
to choose $s \in L$ of odd order $\ppd(r,mf)$, where $m= 6$, $12$, $12$, $9$, $18$, and $30$, respectively. In these cases, $|\CB_L(s)| = \Phi_m(q)$ is odd, where $\Phi_m(t) \in \ZZ[t]$ is the $m^{\mathrm{th}}$ 
cyclotomic polynomial in variable $t$. 

In fact, for $S$ of type $E_6(q)$ or $\tw2 E_6(q)$, by \cite[Lemma 2.3]{MT} we can also choose a semisimple element $s_2 \in F_4(q) < L$ of odd order $\ppd(r,12f)$, for which we have $|\CB_L(s_2)|= \Phi_{12}(q)\Phi_3(q)$, 
respectively $\Phi_{12}(q)\Phi_6(q)$. 
\end{proof}

Let $q\geq 2$ be any prime power. We will work with the unipotent characters $\chi^\lambda$ 
of $G=\GL_n(q)$ (which contain $\ZB(\GL_n(q))$ in their kernels and hence will be viewed as 
characters of $\PGL_n(q)$ which 
restrict irreducibly to $\PSL_n(q)$) labeled by the partitions $\lambda=(n-1,1)$ when $n \geq 2$, 
and $\lambda = (n-2,2)$, $(n-2,1^2)$ when $n \geq 4$. Denote
$$\pb=\chi^{(n-1,1)},~~\pc=\chi^{(n-2,2)},~~\pd=\chi^{(n-2,1^2)}.$$
In what follows we will freely view these characters both over $G$ and $G/Z$, and similarly for subgroups
containing $Z$.
Given a chosen prime $\ell$ and a finite group $X$, we will denote by $\chi^\circ$ the restriction to $\ell'$-elements of an ordinary character $\chi$
of $X$. With the above notation, we can prove

\begin{prop}\label{gl-unip}
The following statements hold for $S:=\PSL_n(q)$ with $n \geq 3$, $A:=\Aut(S)$, and $P \in \Syl_2(A)$.
\begin{enumerate}[\rm(i)]
\item Suppose $n \geq 3$. Then $\pb_S$ extends to a strongly real character $\tilde\pb$ of $A$. 
If in addition $2 \nmid nq$, then $\tilde\pb$ can be chosen to contain $1_P$ and vanishing on some $2$-element $x \in S$.
\item Suppose $n \geq 4$. Then $\pc_S$ extends to a character $\tilde\pc$ of $A$ which is afforded by a $\QQ$-representation.
If in addition $4|n$ and $2 \nmid q$, then $\tilde\pc$ can be chosen to contain $1_P$ and vanishing on some $2$-element $x \in S$.
\item Suppose $n \geq 4$. Then $\pd_S$ extends to a character $\tilde\pd$ of $A$ which is afforded by a $\QQ$-representation.
If in addition $n \equiv 2 \pmod4$ and $2 \nmid q$, then $\tilde\pd$ can be chosen to contain $1_P$ 
and vanishing on some $2$-element $x \in S$.
\end{enumerate}
\end{prop}

\begin{proof}
Let $q=r^f$ for a prime $r$ and let 
$V = \langle e_1, e_2, \ldots,e_n \rangle_{\FF_q}$ be the natural module for $G:=\GL_n(q)$. In the basis $(e_1, \ldots,e_n)$, consider
the transformation $\sigma:\sum^n_{i=1} x_ie_i \mapsto \sum^n_{i=1}x_i^re_i$ and set $\Gamma:=\langle G,\sigma\rangle$,
$Z:=\ZB(G)$, $B:=\Gamma/Z$. We also consider the transpose-inverse automorphism $\tau:g \mapsto \tw tg^{-1}$ for any 
$g \in G$ written in this basis.

\smallskip
(i) The first statement is proved in \cite[Lemma 6.2]{T}. More precisely, $G$ and $\Gamma$ act doubly transitively on the set $\Omega_1$ 
of $1$-dimensional $\FF_q$-subspaces of $V$, and  
$\tilde\pb_{B}+1_B$ is chosen to be the permutation character of this action.
Taking $\ell$ to be a prime divisor of $(q^n-1)/(q-1)$, we see by 
\cite[Proposition 3.1]{GT} that $(\tilde\pb^\circ)_{G/Z}$ is the sum of $1_{G/Z}$ and an irreducible $\ell$-Brauer character.   
Since $\ell$ divides $|\Omega_1|$, $(\tilde\pb^\circ)_B$ also contains $1_B$. It follows that $(\tilde\pb^\circ)_B$ contains $1_B$ with multiplicity one.

Now assume that $2 \nmid nq$. Then $\pb(1)=(q^n-q)/(q-1) \geq 12$ is even, and $\ell > 2$. Applying Lemma \ref{ext}(iii) to the normal 
subgroup $B$ of index $2$ in $A$, and multiplying $\tilde\pb$ by a linear character of $A/B$ of order $2$ if necessary, we achieve that
$(\tilde\pb)_P$ contains $1_P$. Next we choose $\varep=\varep_{n-1} \in \overline{\FF_q}$ of order $|\varep|=(q^{n-1}-1)_2 \geq (q^2-1)_2 > (q-1)_2$.
Then there is a $2$-element $x \in S$ whose preimage in $\SL_n(q)$ has spectrum
$$\bigl\{\varep,\varep^q, \ldots,\varep^{q^{n-2}},\varep^{-(q^{n-1}-1)/(q-1)}\bigr\}$$
on $V$. As $x$ fixes a unique $1$-space of $V$, $\tilde\pb(x)=0$ as desired. 

\smallskip
(ii) Consider the subgroup $H:=\mathrm{Stab}_G(U_1,W) \cong \GL_2(q) \times \GL_{n-2}(q)$, where
$U_1:=\langle e_1,e_2 \rangle_{\FF_q}$ and $W:= \langle e_3, e_4, \ldots,e_n \rangle_{\FF_q}$. Then $H$ has $3$ orbits on the set
$\Omega_1$ of $1$-dimensional $\FF_q$-subspaces of $V$, with representatives $\langle e_1\rangle_{\FF_q}$, $\langle e_3\rangle_{\FF_q}$,
and $\langle e_1+e_3\rangle_{\FF_q}$. The same holds for $\hat{H}:= \langle H, \sigma \rangle$ and for $H \cap L$, where $L:=\SL_n(q)$.

Next we count the $H$-orbits on the set $\Omega_2$ of $2$-dimensional $\FF_q$-subspaces of $V$.
For a representative $U$ of such an orbit, we can choose $U=U_1$ if $U \subseteq U_1$, and $U=U_2:=\langle e_3,e_4\rangle_{\FF_q}$ if 
$U \subseteq W$. If $0 \neq U \cap U_1 \neq U$, we can choose $U=\langle e_1,ae_2+v\rangle_{\FF_q}$ with $0 \neq v \in W$ and $a \in \FF_q$;
so with $a=0$ we can choose $U=U_3:=\langle e_1,e_3\rangle_{\FF_q}$ when $0 \neq U \cap W$ and with $a \neq 0$ 
we can choose $U=U_4:=\langle e_1,e_2+e_3\rangle_{\FF_q}$ when $U \cap W=0$. Assume now that $U \cap U_1=0$ and $U \not\subseteq W$. Then we can choose $U=U_5:=\langle e_1+e_3,e_4\rangle_{\FF_q}$ when $0 \neq U \cap W$, and 
$U=U_6:=\langle e_1+e_3,e_2+e_4\rangle_{\FF_q}$ when $U \cap W=0$. Thus $H$ has $6$ orbits on $\Omega_2$,
and the same also holds 
true for $\hat{H}$, since $\sigma$ preserves $U_1$ and $W$. We also note that 
$$[H:\mathrm{Stab}_H(U_i)]=[(H \cap L):\mathrm{Stab}_{H \cap L}(U_i)]$$
for all $1 \leq i \leq 5$, and also for $i=6$ except when $n=4$ and $2 \nmid q$, where
$$[H:\mathrm{Stab}_H(U_6)]=2[(H \cap \SL_4(q)):\mathrm{Stab}_{H \cap \SL_4(q)}(U_6)].$$
It follows that the number of $(H \cap L)$-orbits on $\Omega_2$ is $6$, unless $n=4$ and $2 \nmid q$ when it is $7$.   

By \cite[Lemma 5.1]{GT}, $1_G+\pb$ is the permutation character of $G$ on $\Omega_1$, and $1_G+\pb+\pc$ is the permutation character of 
$G$ on $\Omega_2$. The permutation actions on $\Omega_1$ and $\Omega_2$ also extend to $\Gamma$, 
with point stabilizers $\hat{P}_1$ and $\hat{P}_2$. As $\hat{P}_2$ has two orbits on $\Omega_1$, the inner product of the corresponding 
permutation characters over $\Gamma$ is $2$. 
Since $1 < \pb(1) < \pc(1)$, it follows that these permutation characters are $1_\Gamma + \hat\pb$ and $1_\Gamma+\hat\pb+\hat\pc$,
where $\hat\pb$ is a rational extension of $\pb$ (which is equal to $\tilde\pb_B$, where $\tilde\pb$ is as defined in (i)), 
and $\hat\pc$ is a rational extension of 
$\pc$ to $B=\Gamma/Z$. 
The above computations show that 
$$3=[(1_\Gamma+\hat\pb+\hat\pc)_{\hat H},1_{\hat H}]-[(1_\Gamma+\hat\pb)_{\hat H},1_{\hat H}] = [(\hat\pc)_{\hat H},1_{\hat H}]=[\hat\pc,(1_{\hat H})^\Gamma]$$
and 
 $$4 \geq [(1_L+\pb_L+\pc_L)_{H\cap L},1_{H \cap L}]-[(1_L+\pb_L)_{H \cap L},1_{H \cap L}] = [(\pc)_{H \cap L},1_{H \cap L}]=[\pc_L,((1_H)^G)_L].$$
Now consider the $A$-character $\rho:= (1_{\langle H/Z,\sigma,\tau\rangle})^A$ which restricts to $(1_{\hat H})^\Gamma$ and hence contains 
$3\hat\pc$ on $G$. If $\hat\pc$ is not $A$-invariant, then $\rho_B$ also contains $3\hat\pc'$ for some $A$-conjugate $\hat\pc' \neq \hat\pc$. On the other hand,
$\pc_S$ is $\Aut(S)$-invariant by \cite[Theorem 2.5]{Ma}, so $(\hat\pc')_S = \hat\pc_S$, and thus $[\pc_L,((1_H)^G)_L] \geq 6$, a contradiction.
So $\hat\pc$ is $A$-invariant.
Applying Lemma \ref{index2}, we see that $\hat\pc$ has an extension $\tilde\pc$ to $A$ which is afforded by a $\QQ A$-module.

Now assume that $4|n$ and $2 \nmid q$. Then $\pc(1)=q^2(q^n-1)(q^{n-3}-1)/(q-1)(q^2-1)$ is even. Furthermore, choosing $\ell$ to be 
an (odd) prime divisor of $(q^{n-1}-1)/(q-1)$, we see by \cite[Proposition 3.1]{GT} that $\pc^\circ$ contains $1_G$ with multiplicity one. Applying Lemma \ref{ext}(iii) to the normal 
subgroup $G/Z$ of $A$, and multiplying $\tilde\pc$ by a linear character of $A/(G/Z)$ of order $2$ if necessary, we achieve that
$(\tilde\pc)_P$ contains $1_P$.
Next, if $n=4$ we choose $\xi \in \overline{\FF_q}$ of order $|\xi|=(q^2-1)_2 > (q-1)_2$, so that $\xi^{q+1}$ has order $(q-1)_2 > 1$, and 
find a $2$-element $x \in S$ whose preimage in $\SL_n(q)$ has spectrum
$$\bigl\{\xi,\xi^q,1,\xi^{-(q+1)}\bigr\}$$
on $V$. Then, $x$ has exactly two fixed points on $\Omega_1$ and exactly two fixed points on $\Omega_2$, whence
$\pb(x)=1$ and $\pc(x)=0$. Assume $4|n \geq 8$. 
Then we choose $\varep=\varep_{n-4} \in \overline{\FF_q}$ of order $|\varep|=(q^{n-4}-1)_2 \geq (q^4-1)_2 > (q^2-1)_2$, 
and take $\zeta=\varep^{-(q^{n-4}-1)/(q^4-1)} \in \overline{\FF_q}$ of order $|\zeta|=(q^4-1)_2 > (q^2-1)_2$.
Now there is a $2$-element $x \in S$ whose preimage in $\SL_n(q)$ has spectrum
$$\bigl\{\varep,\varep^q, \ldots,\varep^{q^{n-5}},\zeta,\zeta^q,\zeta^{q^2},\zeta^{q^3}\bigr\}$$
on $V$. Then, $x$ has no fixed point on $\Omega_1$ and no fixed point on $\Omega_2$, whence
$\pb(x)=-1$ and $\pc(x)=0$. 

\smallskip
(iii) Consider the subgroup 
$K:=\mathrm{Stab}_G(\langle e_1\rangle_{\FF_q},\langle e_2, \ldots,e_n\rangle_{\FF_q}) \cong \GL_1(q) \times \GL_{n-1}(q)$. As shown in the proof
of \cite[Proposition 5.5]{NT1}, $\pd$ has multiplicity one in $(1_K)^G$;
in fact it is the only irreducible constituent of degree $\pd(1)$. Since the restriction to $G$ of $(1_{\langle K/Z,\sigma,\tau\rangle})^A$ is
equal to $(1_K)^G$, it follows that $\pd$ is $A$-invariant. Hence $\pd$ has a rational-valued extension $\tilde\pd$ to $A$, which is an irreducible
constituent of $(1_{\langle K/Z,\sigma,\tau\rangle})^A$ of multiplicity one, by Lemma \ref{ext}(i). Now the
Schur index $m_\QQ(\tilde\pd)$ must be one by \cite[Corollary (10.2)(c)]{Is}, so $\tilde\pd$ is afforded by a $\QQ A$-module.

Next assume that $4|(n-2)$ and $2 \nmid q$. Then $\pd(1)=(q^n-q)(q^n-q^2)/(q-1)(q^2-1)$ is even. Furthermore, choosing $\ell$ to be 
an odd prime divisor of $(q^n-1)/(q-1)$, we see by \cite[Proposition 3.1]{GT} that there are irreducible $\ell$-Brauer characters 
$\varphi$ and $\psi$ of $G$ trivial at $Z$ such that $\pb^\circ=1_G + \varphi$ and $\pd^\circ=\varphi+\psi$. Since 
$1 < \varphi(1) < \psi(1)$ and $\tilde\pb$ (as specified in part (i)) and $\tilde\pd$ are real-valued, 
it follows from Lemma \ref{ext}(iii) that there are real-valued extensions $\varphi_1$ and 
$\varphi_2$ of $\varphi$ to $A$ and $\psi_1$ of $\psi$ to $A$, and a linear character $\lambda_1$ of $A$ such that
$$\tilde\pb^\circ = \lambda_1+\varphi_1,~(\lambda_1)_B = 1_B,~\tilde\pd^\circ=\varphi_2 +\psi_1.$$
As $A/(G/Z)$ is abelian, $\varphi_2=\varphi_1\nu$ for some real-valued, hence of order $\leq 2$, linear character $\nu$ of $A/(G/Z)$. 
Changing $\tilde\pd$ to $\tilde\pd\nu$, we may assume that $\varphi_2=\varphi_1$. In particular, $(\varphi_2)_Q = \tilde\pb_Q-1_Q$
for $Q := P \cap B \in \Syl_2(B)$. 

The assumption $4|(n-2)$ implies that the index $(q^n-1)(q^{n-1}-1)/(q-1)(q^2-1)$ of 
$H = \GL_2(q) \times \GL_{n-2}(q)$ in $G$ is odd, and similarly for the index of $\hat{H}/Z = \langle H,\sigma \rangle/Z$ in 
$B$. Hence we may assume that $Q < \hat{H}/Z$. But $\hat{H}$ has $3$ orbits on $\Omega_1$, so 
$[\hat\pb_Q,1_Q] \geq 2$. Recalling $\hat\pb=\tilde\pb_B$, we deduce that  
$[\tilde\pb_Q,1_Q] \geq 2$, and hence $(\varphi_2)_Q$ contains $1_Q$.
Applying Lemma \ref{ext}(iii) to the normal 
subgroup $B$ of index $2$ of $A$ (and the $A$-invariant character $(\varphi_2)_B$), 
and multiplying $\tilde\pd$ by a linear character of $A/B$ of order $2$ if necessary, we achieve that
$(\tilde\pd)_P$ contains $1_P$. We also choose $\varep=\varep_{n-2} \in \overline{\FF_q}$ of order 
$|\varep|=(q^{n-2}-1)_2 \geq (q^4-1)_2 > (q^2-1)_2$, 
and take $\xi=\varep^{-(q^{n-2}-1)/(q^2-1)} \in \overline{\FF_q}$ of order $|\xi|=(q^2-1)_2 > (q-1)_2$.
Now there is a $2$-element $x \in S$ whose preimage in $\SL_n(q)$ has spectrum
$$\bigl\{\varep,\varep^q, \ldots,\varep^{q^{n-3}},\xi,\xi^q\bigr\}$$
on $V$. Then, $x$ has no fixed point on $\Omega_1$ and a unique fixed point on $\Omega_2$, whence
$\pb(x)=-1$ and $\pc(x)=1$. This element $x$ also has no fixed point while acting on the set of partial flags
$0 < V_1 < V_2 < V$ of $V$, with $\dim_{\FF_q}(V_i)=i$. By \cite[Lemma 5.1(ii)]{GT}, $G$ acts on the latter set with
permutation character $1_G +2\pb+\pc+\pd$, which implies that $\pd(x)=0$.  
\end{proof}

\begin{prop}\label{gu-unip}
Let $q$ be any odd prime power, and let $S:=\PSU_n(q)$ and $A:=\Aut(S)$.
If $n \geq 4$, then $A$ admits a character $\chi$ afforded by a $\QQ$-representation, such that $\chi_S \in \Irr(S)$, 
$\chi_P$ contain $1_P$ for $P \in \Syl_2(A)$, and $\chi(h)=0$ for some $2$-element $h \in S$. If $n=3$, then $S$ admits an irreducible character $\alpha$ and a $2$-element $h \in S$ such that 
$\alpha^\varsigma(h)=0$ for all $\varsigma \in \Aut(S)$ and $\alpha$ extends to 
a strongly real character of $I=I_{\Aut(S)}(\alpha)$ which lies above $1_Q$ for $Q \in \Syl_2(I)$.
\end{prop}

\begin{proof}
(a) First we consider the case $n \geq 4$.
Let $V = \FF_{q^2}^n$ be endowed with a non-degenerate Hermitian form and let $G = \GU(V) \cong \GU_n(q)$ be the corresponding
general unitary group. Consider the permutation action of $G$ on the set $\Omega$ of singular $1$-dimensional subspaces of $V$. Then 
it is well known, see e.g. \cite[\S2]{ST} that the permutation character of this action is $1_G + \pphi+\ppsi$, where $\pphi$ and $\ppsi$ are 
unipotent characters of distinct degrees $>1$ (listed in \cite[Table 2]{ST}), both trivial at $Z:=\ZB(G)$ and irreducible over $S=[G/Z,G/Z]$. 
Fixing an orthonormal basis $(e_1, \ldots,e_n)$ of $V$ and considering the map
$\sigma:\sum^n_{i=1}x_ie_i \mapsto \sum^n_{i=1}x_i^re_i$, $x_i \in \FF_{q^2}$, for $q=r^f$ and $r|q$ a prime, we see that this 
action also extends to $A= \Gamma/Z$, where $\Gamma:=\langle G,\sigma\rangle$. It follows from Lemma \ref{ext}(i) that 
the character of this $A$-action is $1_A + \tilde\pphi+\tilde\ppsi$, where $\tilde\pphi$ and $\tilde\ppsi$ are rational extensions of 
$\pphi$ and $\ppsi$, both afforded by $\QQ$-representations.

\smallskip
(a1) The desired character $\chi$ will be chosen to be either $\tilde\pphi$ or $\tilde\ppsi$, depending on the congruence $n \pmod4$, and possibly 
multiplying by a linear character $\lambda$ of $A/S$ of order $2$. First suppose 
that $4|n=2m$. Choosing $\ell = \ppd(q,4m-2)$ (which exists by \cite{Zs}), we see that the character $\ppsi_S$ of degree
$q^3(q^{2m-2}-1)(q^{2m-1}+1)/(q+1)(q^2-1)$ has $\ell$-defect zero. On the other hand, as shown in \cite{L1}, one of $\pphi_S$ and $\ppsi_S$ 
contains $1_S$ with multiplicity one in its restriction to $\ell'$-classes, so this must be $\pphi_S$, of degree
$(q^n-1)(q^{n-1}+q^2)/(q+1)(q^2-1)$. Applying Lemma \ref{ext}(iii) to the character $\pphi_S$ of $S \lhd A$, we
see that some $\chi \in \{\tilde\pphi,\tilde\pphi\lambda\}$ contains $1_P$. To prove the existence of the desired $2$-element $h \in S$, we will use 
the dual pair $L \times T$ as in \cite[\S6]{LBST}, where $L = \SU(V) \cong \SU_n(q)$, $W=\FF_{q^2}^2$,  and $T = \GU(W) \cong \GU_2(q)$, and  
the identification $\pphi_S = D_{1_T} -1_S$, see  \cite[Propositions 6.3, 6.5]{LBST}. For any $\al \in \Irr(T)$, the $L$-character $D_\al$ can be computed using the formula 
$$D_\al(g) = \frac{1}{|T|}\sum_{x \in T}\overline{\al(x)}(-q)^{d(xg)},$$
where $d(xg)$ is the dimension of the $xg$-fixed point subspace of $xg$ acting on $W \otimes_{\FF_{q^2}}V$ (with $x$ acting on $W$ and
$g$ acting on $V$), see \cite[Lemma 5.5]{LBST}. Our element $h$ will be the image in $S$ of a suitably chosen $g \in L$.

First suppose that $n=4$, and choose $\xi \in \overline{\FF_q}$ of order $|\xi|=(q^2-1)_2 > (q-1)_2$, so that $\xi^{q-1}$ has order $(q+1)_2 > 1$. Then we can 
find a $2$-element $g \in L$ has spectrum
$$\bigl\{\xi,\xi^{-q},1,\xi^{q-1}\bigr\}$$
on $V$. Note that $d(xg)$ is $2$ if $x \in T$ is conjugate to $\diag(\xi^{-1},\xi^q)$, $\diag(1,\xi^{1-q})$, $I_2$ or $\xi^{1-q}I_2$, $1$ if $x \in T$ is conjugate to
$\diag(a,b)$ with $a \in \{1,\xi^{1-q}\}$ and $b \in \FF_{q^2} \smallsetminus \{1,\xi^{1-q}\}$, $b^{q+1}=1$, or $\begin{pmatrix}1 & 1\\0 &1\end{pmatrix}$
or $\xi^{1-q}\begin{pmatrix}1 & 1\\0 &1\end{pmatrix}$, and $0$ otherwise. It follows that
$$|T|\pphi(g) = (q^2-1)\bigl(q(q+1)+q(q-1)+2 \bigr)-(q+1)\bigl( 2q(q-1)(q-1)+2(q^2-1)\bigr) = 0$$
as desired. Next assume $4|n \geq 8$. 
Then we choose $\varep=\varep_{n-4} \in \overline{\FF_q}$ of order $|\varep|=(q^{n-4}-1)_2 \geq (q^4-1)_2 > (q^2-1)_2$, 
and take $\zeta=\varep^{-(q^{n-4}-1)/(q^4-1)}$ of order $|\zeta|=(q^4-1)_2 > (q^2-1)_2$.
Now there is a $2$-element $g \in L$ has spectrum
$$\bigl\{\varep,\varep^{-q}, \varep^{q^2},\ldots,\varep^{-q^{n-5}},\zeta,\zeta^{-q},\zeta^{q^2},\zeta^{-q^3}\bigr\}$$
on $V$. None of these eigenvalues belong to $\FF_{q^2}$, so $d(xg)=0$ for all $x \in T$, and hence
$$|T|D_{1_T}(g) = \sum_{x \in T}1 = |T|,$$
i.e. $\pphi(g)=0$.

\smallskip
(a2) Next suppose that $n=2m$ with $2 \nmid m \geq 3$. Choosing $\ell = \ppd(q,m)$, we see that the character $\pphi_S$ of degree
$q^2(q^{2m}-1)(q^{2m-3}+1)/(q+1)(q^2-1)$ has $\ell$-defect zero. Together with the main result of \cite{L1}, this implies that $\ppsi_S$, of degree
$(q^n+q)(q^n-q^2)/(q+1)(q^2-1)$, 
contains $1_S$ with multiplicity one in its restriction to $\ell'$-classes. Applying Lemma \ref{ext}(iii) to the character $\ppsi_S$ of $S \lhd A$, we
conclude that some $\chi \in \{\tilde\ppsi,\tilde\ppsi\lambda\}$ contains $1_P$.

By \cite[Proposition 6.3]{LBST}, $\ppsi_S$ is identified as $D_{\St}-1_S$ where $\St$ is the Steinberg character of $T$.
Choose $\varep=\varep_{n-2} \in \overline{\FF_q}$ of order 
$|\varep|=(q^{n-2}-1)_2 \geq (q^4-1)_2 > (q^2-1)_2$, 
and take $\xi=\varep^{-(q^{n-2}-1)/(q^2-1)}$ of order $|\xi|=(q^2-1)_2 > (q+1)_2$.
Now there is a $2$-element $g \in L$ has spectrum
$$\bigl\{\varep,\varep^{-q},\varep^{q^2}, \ldots,\varep^{-q^{n-3}},\xi,\xi^{-q}\bigr\}$$
on $V$. Recall that $\St(x)=0$ unless $x \in T$ is semisimple. If $x \in T$ is semisimple, then $d(xg)$ is $2$ if 
$x$ is conjugate to $\diag(\xi^{-1},\xi^q)$, and $0$ otherwise. It follows that
$$|T|D_{\St}(g)=q(q+1)q^2+q(q+1)-\frac{q(q+1)}{2}q(q-1)+\frac{q^2-q-4}{2}q(q+1) = |T|,$$
yielding $\ppsi(g)=0$.

\smallskip
(a3) Assume now that $n=2m+1$ with $2 \nmid m \geq 3$. Choosing $\ell = \ppd(q,m)$, we see that the character $\pphi_S$ of degree
$q^3(q^{2m}-1)(q^{2m-1}+1)/(q+1)(q^2-1)$ has $\ell$-defect zero. Again using \cite{L1}, this implies that $\ppsi_S$, of degree
$(q^n+1)(q^{n-1}-q^2)/(q+1)(q^2-1)$, 
contains $1_S$ with multiplicity one in its restriction to $\ell'$-classes. Applying Lemma \ref{ext}(iii) to the character $\ppsi_S$ of $S \lhd A$, we
conclude that some $\chi \in \{\tilde\ppsi,\tilde\ppsi\lambda\}$ contains $1_P$.

As in (a1), $\ppsi_S$ is identified as $D_{1_T}-1_S$. Choose $\varep=\varep_{n-3} \in \overline{\FF_q}$ of order 
$|\varep|=(q^{n-3}-1)_2 \geq (q^4-1)_2 > (q^2-1)_2$, 
and take $\xi=\varep^{-(q^{n-3}-1)/(q^2-1)}$ of order $|\xi|=(q^2-1)_2 > (q+1)_2$.
Then there is a $2$-element $g \in L$ has spectrum
$$\bigl\{\varep,\varep^{-q},\varep^{q^2}, \ldots,\varep^{-q^{n-4}},\xi,\xi^{-q},1\bigr\}$$
on $V$. Note that $d(xg)$ is $2$ if $x \in T$ is conjugate to $\diag(\xi^{-1},\xi^q)$ or $I_2$, $1$ if $x \in T$ is conjugate to
$\diag(1,b)$ with $b \in \FF_{q^2}$ and $b^{q+1}=1 \neq b$, or $\begin{pmatrix}1 & 1\\0 &1\end{pmatrix}$, and $0$ otherwise. It follows that
$$|T|\ppsi(g) = (q^2-1)\bigl(q(q+1)+1 \bigr)-(q+1)\bigl( q(q-1)q+(q^2-1)\bigr) = 0$$
as desired.

\smallskip
(a4) Next suppose that $n=2m+1$ with $2|m$. Choosing $\ell = \ppd(q,4m+2)$, we see that the character $\ppsi_S$ of degree
$q^2(q^{2m+1}+1)(q^{2m-2}-1)/(q+1)(q^2-1)$ has $\ell$-defect zero. Using \cite{L1}, this implies that $\pphi_S$, of degree
$(q^n-q)(q^n+q^2)/(q+1)(q^2-1)$, 
contains $1_S$ with multiplicity one in its restriction to $\ell'$-classes. Applying Lemma \ref{ext}(iii) to the character $\pphi_S$ of $S \lhd A$, we
conclude that some $\chi \in \{\tilde\pphi,\tilde\pphi\lambda\}$ contains $1_P$.

As in (a2), $\pphi_S$ is identified as $D_{\St}-1_S$ where $\St$ is the Steinberg character of $T$.
Choose $\varep=\varep_{n-1} \in \overline{\FF_q}$ of order 
$|\varep|=(q^{n-1}-1)_2 \geq (q^4-1)_2 > (q^2-1)_2$, 
and take $\xi=\varep^{(q^{n-1}-1)/(q+1)}$ of order $|\xi|=(q+1)_2 > 1$.
Then there is a $2$-element $g \in L$ has spectrum
$$\bigl\{\varep,\varep^{-q},\varep^{q^2}, \ldots,\varep^{-q^{n-2}},\xi\bigr\}$$
on $V$. Again, $\St(x)=0$ unless $x \in T$ is semisimple. If $x \in T$ is semisimple, then $d(xg)$ is $2$ if 
$x=\xi^{-1}I_2$, $1$ if $x$ is conjugate to $\diag(b,\xi^{-1})$ where $b \in \FF_{q^2} \smallsetminus \{\xi^{-1}\}$ and
$b^{q+1}=1$, and $0$ otherwise. It follows that
$$|T|D_{\St}(g)=q\cdot q^2+q(q+1)+q \cdot q(q-1)q+q \cdot q - \frac{q(q-1)}{2}q(q-1) + \frac{q^2-q-2}{2}q(q+1) = |T|,$$
yielding $\pphi(g)=0$.

\smallskip
(b) It remains to handle the case $n=3$. Consider the character $\chi=\chi^{(0,1,q)}_{(q-1)(q^2-q+1)}$ of $G=\GU_3(q)$ 
in the notation of \cite[p. 31]{E}.
Then it is easy to check that $\chi$ is trivial at $Z:=\ZB(G)$, real-valued, restricts irreducibly to $L:= \SU_3(q)$, by using the character table of $L$
given in \cite{Geck} and that $(q+1)/3 > 1$, and vanishes at any $2$-element $x$ of $L$ of order $(q^2-1)_2$. 
Direct computation shows that $[\chi_H,1_H]=1$ for $H := \GU_2(q) \times \GU_1(q)$. Thus $\chi$ occurs with 
multiplicity one in $(1_H)^G$, which is the permutation character of $G$ acting on the set $\Delta$ of non-singular $1$-dimensional subspaces of 
the natural Hermitian space $V:=\FF_{q^2}^3$ for $G$. Using an orthonormal basis of $V$ to define the group
$\Gamma:=\langle G,\sigma \rangle \cong \Gamma U(V)$ as in (a), one can verify (by checking the values at the 
class $C_4^{(1)}$, in the notation of \cite[Table 3.1]{Geck}) that $I:=I_A(\chi_S)=I_A(\chi) = \tilde G/Z$, 
where $\tilde G := \langle G,\sigma^f \rangle \cong G \cdot 2$ if $q=r^f$. The permutation action of $G$ on $\Delta$ extends to 
$\tilde G$. By Lemma \ref{ext}(ii), $\chi$ extends to a strongly real character $\tilde\chi$ of $\tilde G$ which can be viewed as a character of 
$I_A(\chi_S)$. Note that
$[G:H]=q^2-q+1$ is odd, so a Sylow $2$-subgroup $R$ of $G$ lies in $H$, and hence $\chi_R$ contains $1_R$. Applying Lemma \ref{ext}(iii)
to the normal subgroup $G$ of index $2$ in $\tilde G$, we conclude that $\tilde\chi\lambda$ lies over $1_Q$ for some linear character $\lambda$ of 
$\tilde G/G$ and $Q \in \Syl_2(I)$.
\end{proof}

\begin{prop}\label{typeD-unip}
Let $q$ be any odd prime power, and let $S:=\PO^\eps_{2n}(q)$ with $2 \nmid n \geq 5$, $\eps=\pm$, and $A:=\Aut(S)$.
Then $A$ admits a character $\chi$ afforded by a $\QQ$-representation, such that $\chi_S \in \Irr(S)$, 
$\chi_P$ contain $1_P$ for $P \in \Syl_2(A)$, and $\chi(h)=0$ for some $2$-element $h \in S$.
\end{prop}

\begin{proof}
Let $V = \FF_q^{2n}$ be endowed with a non-degenerate quadratic form, constructed as in \cite[Proposition 2.5.12]{KlL},
and let $G = \GO(V) \cong \GO^\pm_{2n}(q)$ be the corresponding
orthogonal group, so that $S=[G,G]/(\ZB(G) \cap [G,G])$. Next, one can define a finite group 
$\Gamma = \langle \CO(V),\phi \rangle \rhd G$
where the conformal group $\CO(V)$ consists of all $x \in \GL(V)$ that preserve the form on $V$ up to a scalar and 
$\phi$ is a semilinear transformation of $V$, such that $A=\Gamma/Z$ where $Z :=\ZB(\GL(V))$, see \cite[Theorem 2.1.4]{KlL}.
Consider the permutation action of $G$ on the set $\Omega$ of singular $1$-dimensional subspaces of $V$. Then 
it is well known, see e.g. \cite[\S2]{ST} that the permutation character $\rho$ 
of this action is $1_G + \pphi+\ppsi$, where $\pphi$ and $\ppsi$ are 
unipotent characters of degrees $(q^n-\eps)(q^{n-1}+\eps q)/(q^2-1)$, respectively $(q^{2n}-q^2)/(q^2-1)$, 
both trivial at $Z$ and irreducible over $S$. 
This permutation action also extends to $\Gamma/Z$. It follows from Lemma \ref{ext}(i) that 
the character of this $A$-action is $1_A + \tilde\pphi+\tilde\ppsi$, where $\tilde\pphi$ and $\tilde\ppsi$ are rational extensions of 
$\pphi$ and $\ppsi$, both afforded by $\QQ$-representations.

The desired character $\chi$ will be chosen to be $\tilde\ppsi$, possibly 
multiplying by a linear character $\lambda$ of $A/S$ of order $2$. Choosing $\ell = \ppd(q,n)$ when $\eps=+$ and 
$\ell=\ppd(q,2n)$ when $\eps=-$, (which exists by \cite{Zs}), we see that the character $\pphi_S$ 
has $\ell$-defect zero. On the other hand, as shown in \cite{L2}, one of $\pphi_S$ and $\ppsi_S$ 
contains $1_S$ with multiplicity one in its restriction to $\ell'$-classes, so this must be $\ppsi_S$. 
Applying Lemma \ref{ext}(iii) to the character $\ppsi_S$ of $S \lhd A$, we
see that some $\chi \in \{\tilde\ppsi,\tilde\ppsi\lambda\}$ lies over $1_P$. 

To prove the existence of the desired $2$-element $h \in S$, we will use 
the dual pair $L \times T$ as in \cite[\S5]{LBST}, where $L = \Omega(V) \cong \Omega^\eps_{2n}(q)$, $W=\FF_q^2$,  and 
$T = \Sp(W) \cong \Sp_2(q)$, and  
the identification $\ppsi_S = D_\St -1_S$ where $\St$ is the Steinberg character of $T$, see  \cite[Proposition 5.7]{LBST}.  
By \cite[Lemma 5.5]{LBST}, the $L$-character $D_\St$ can be computed using the formula 
$$D_\St(g) = \frac{1}{|T|}\sum_{x \in T}\St(x)\omega(xg),$$
where $\omega$ is a (reducible) Weil character of $\Sp(V \otimes_{\FF_q}W) \cong \Sp_{4n}(q)$, and it is known that 
$\omega(xg)= \pm q^{d(xg)/2}$ when both $x$ and $g$ are semisimple, 
where $d(xg)$ is the dimension of the $xg$-fixed point subspace of $xg$ acting on 
$V \otimes_{\FF_{q}}W$ (with $x$ acting on $W$ and $g$ acting on $V$).
As in \cite[\S5]{LBST}, we fix a basis $(v_1, \ldots,v_{2n})$ of $V$, in which the quadratic form on $V$ has Gram matrix
$\diag(1,1,\ldots,1,\gamma)$ for a suitable $\gamma \in \FF_q^\times$, and a basis $(e,f)$ of $W$ such that the symplectic form on $W$ takes value
$1$ at the pair $(e,f)$. This ensures that 
$$U:=\langle e \otimes v_1, \ldots, e\otimes v_{2n-1},e\otimes v_{2n}\rangle_{\FF_q},~~
    U':=\langle f \otimes v_1, \ldots, f\otimes v_{2n-1},f\otimes \gamma^{-1}v_{2n}\rangle_{\FF_q}$$
form a complementary pair of maximal totally isotropic subspaces in $V \otimes_{\FF_q}W$, so that 
$$\mathrm{Stab}_{\Sp(V \otimes _{\FF_q}W)}(U,U') \cong \GL_{2n}(q).$$
Choose $\varep \in \overline{\FF_q}$ of order 
$|\varep|=(q^{n-1}-1)_2 \geq (q^2-1)_2 > (q \pm 1)_2$, 
and take $\xi=\varep^{(q^{n-1}-1)/(q-\eps)}$ of order $|\xi|=(q-\eps)_2$.
Then we choose the desired element $h$ to be the image in $S$ of any $2$-element $g \in L$ that belongs to
$\GL_{n-1}(q) \times \mathsf{C}_{q-\eps} < \SO^+_{2n-2}(q) \times \SO^\eps_2(q)$ and has spectrum
$$\bigl\{\varep,\varep^{q},\ldots,\varep^{q^{n-2}},\varep^{-1},\varep^{-q},\ldots,\varep^{-q^{n-2}},\xi,\xi^{-1}\bigr\}$$
on $V$. Recall that $\St(x)=0$ unless $x \in T$ is semisimple. If $x \in T$ is semisimple, then $d(xg)$ is $4$ if 
$x=-I_2$ and $\xi=-1$, $2$ if 
$x$ is conjugate to $\diag(\xi,\xi^{-1})$ and $\xi \neq -1$, and $0$ otherwise. Also note that if 
$x \in \Stab_T\bigl(\langle e\rangle_{\FF_q},\langle f \rangle_{\FF_q}\bigr) \cong \mathsf{C}_{q-1}$,
say $x(e) = ae$ for some $a \in \FF_q^\times$,
then $xg$ stabilizes both $U$ and $U'$, and has determinant $a^{2n}\det{g}=a^{2n}$, a square in $\FF_q^\times$. This ensures by 
\cite[(13.3)]{Gr} that 
$$\omega(xg)=q^{d(xg)/2}$$ 
for such elements $x$.

\smallskip
(a) First suppose that $\eps=+$ and $4|(q-1)$. Then in the above construction we can choose $\gamma=1$. The preceding remarks show
that $(\omega(xg),\St(x))=(q,1)$ for the $q(q+1)$ elements $x \in T$ that are conjugate to $\diag(\xi,\xi^{-1})$, 
$(\omega(xg),\St(x))=(1,q)$ for the two elements $x = \pm I_2$, and $(\omega(xg),\St(x))=(\pm 1, \pm 1)$ for the remaining $q(q^2-3q-2)$
semisimple elements $x \in T$. It follows that $|T|D_\St(g)$ is at most
$$q^2(q+1) +2q+q(q^2-3q-2) =2q^2(q-1) <2|T|$$
and at least
$$q^2(q+1) +2q-q(q^2-3q-2) =4q(q+1)> 0.$$
Since $D_\St(g) \in \ZZ$, we conclude that $D_\St(g)=1$, and so $\ppsi(g)=0$.

\smallskip
(b) Next suppose that $4|(q+\eps)$; in particular $\xi = -1$. Then  
$(\omega(xg),\St(x))=(q^2,q)$ for $x=-I_2$, 
$(\omega(xg),\St(x))=(1,q)$ for $x = I_2$, and $(\omega(xg),\St(x))=(\pm 1, \pm 1)$ for the remaining $q(q^2-2q-1)$
semisimple elements $x \in T$. It follows that $|T|D_\St(g)$ is at most
$$q^3 +q+q(q^2-2q-1) =2q^2(q-1) <2|T|$$
and at least
$$q^3 +q-q(q^2-2q-1) =2q(q+1)> 0.$$
Since $D_\St(g) \in \ZZ$, we conclude that $D_\St(g)=1$, and thus $\ppsi(g)=0$.

\smallskip
(c) Finally, assume that $\eps=-$ and $4|(q+1)$. Then  
$(\omega(xg),\St(x))=(1,q)$ for the two elements $x= \pm I_2$, 
$(\omega(xg),\St(x))=(1,1)$ for the $q(q+1)(q-3)/2$ elements $x \in T$ that are conjugate to $\diag(b,b^{-1})$
with $b \in \FF_q^\times \smallsetminus \{\pm 1\}$, $(\omega(xg),\St(x))=(\pm q,-1)$ for the $q(q-1)$ elements $x \in T$ that are conjugate to 
$\diag(\xi,\xi^{-1})$, and $(\omega(xg),\St(x))=(\pm 1, \pm 1)$ for the remaining $q(q-1)(q-3)/2$
semisimple elements $x \in T$. It follows that $|T|D_\St(g)$ is at most
$$2q+q(q+1)(q-3)/2 + q^2(q-1)+q(q-1)(q-3)/2 =2q(q-1)^2 <2|T|$$
and at least
$$2q+q(q+1)(q-3)/2 - q^2(q-1)-q(q-1)(q-3)/2 =-q(q-1)^2> -|T|.$$
Since $D_\St(g) \in \ZZ$, we deduce that $D_\St(g) \in \{0,1\}$, whence $\ppsi(g)\in \{-1,0\}$. Suppose now
that $\ppsi(g)=-1$. Then $\rho(g)=1+\pphi(g)+\ppsi(g)=\pphi(g)$. Note that  $\rho(g)$ is the number of $g$-fixed 
singular $1$-dimensional subspaces in $V$, which is zero since no eigenvalue of $g$ belongs to $\FF_q^\times$. 
It follows that $\pphi(g)=0$, which is a contradiction since $g$ is a $2$-element and 
$\pphi(1) = (q^n+1)(q^{n-1}-q)/(q^2-1)$
is odd. Hence $\ppsi(g)=0$, as stated.
\end{proof}

Now we can complete the proof of Theorem \ref{simple2}, which we restate:

\begin{thm}\label{simple2b}
Let $S$ be a finite non-abelian simple group. Then at least one of the following two statements holds.
  
\begin{enumerate}[\rm(i)]  
\item $S$ admits an irreducible character $\alpha$ and a $2$-element $x \in S$ such that 
$\alpha^\varsigma(x)=0$ for all $\varsigma \in \Aut(S)$ and $\alpha$ extends to 
a strongly real character of $I=I_{\Aut(S)}(\alpha)$ which lies above $1_P$, where $P$ is a Sylow $2$-subgroup of $I$. 

\item If A is an almost simple group with socle $S$ and $A/S$ is a $2$-group, then there exists a character
$\alpha \in \irr S$ of $2$-defect zero  such that 
$\alpha$ extends to a strongly real character $\beta$ of  $I = I_A(\alpha)$ and $\beta$ lies over $1_P$,  where $P$ is a Sylow $2$-subgroup of $I$. 
\end{enumerate} 
\end{thm} 

\begin{proof}
We will assume that $S$ is not isomorphic to any of the simple groups listed in Lemma \ref{small}.  Next, the proof of Proposition \ref{almost-simple-others} shows that Theorem \ref{simple2}(i) holds for $S=\AAA_n$ with $n \geq 7$, and Theorem \ref{simple2} holds when $S$ is  a sporadic
simple group by Lemma \ref{small}. It remains therefore to consider simple groups of Lie type $S$.

\smallskip
(a) First we show that Theorem \ref{simple2}(ii) holds if $S$ is not one of the following groups in odd characteristic: 
$\PSL_n(q)$ or $\PSU_n(q)$ with $n \geq 3$,
$\PO^\pm_{2n}(q)$ with $2 \nmid n \geq 5$, in which case $\beta$ can be chosen to be strongly real and containing $1_P$ with odd multiplicity.
Note that it suffices to show that $S$ admits a
real-valued irreducible character $\theta$ of $2$-defect zero. Indeed, 
$\theta$ is strongly real by Lemma \ref{triv}(i). Applying Lemma \ref{Lgab}(iii) to $(G,N):=(I_A(\theta),S)$, we see
that $\theta$ has a real-valued extension $\beta$ to $I=I_A(\theta)$ which contains $1_P$ with odd multiplicity for $P \in \Syl_2(I)$. 
Since $\beta_S=\theta$ has Frobenius--Schur indicator $1$, the same holds for $\beta$. 

Now, for all the indicated groups $S$, we can follow the proof of Theorem \ref{simple2a} (and using $s=s_2 \in F_4(q)$ 
when $S$ is of type $E_6$ or $\tw2 E_6$). 
Applying \cite[Proposition 3.1(ii)]{TZ}, we see that the chosen element $s$ is real, whence $\theta$ is real as desired.

\smallskip
(b) Now, if $S = \PSL_n(q)$ with $n \geq 3$ and $2 \nmid q$, then Theorem \ref{simple2}(i) holds for $S$ by Proposition \ref{gl-unip}.
If $S = \PSU_n(q)$ with $n \geq 3$ and $2 \nmid q$, then Theorem \ref{simple2}(i) holds for $S$ by Proposition \ref{gu-unip}.
If $S = \PO^\pm_{2n}(q)$ with $2 \nmid n \geq 5$ and $2 \nmid q$, then Theorem \ref{simple2}(i) holds for $S$ by Proposition \ref{typeD-unip}.
\end{proof}

\section{Almost simple groups. II}
In this section we will complete the proof of Theorem \ref{simple1} for $p>2$.

\begin{lem}\label{typeD4-unip}
Let $q$ be any prime power coprime to $3$, and let $S:=\PO^+_{8}(q)$ and $A:=\Aut(S)$.
Then $S$ admits an irreducible character $\alpha$ such that $\alpha$ extends to a character $\chi$ of a 
$\QQ$-representation of $I=I_A(\alpha)$,  
$\chi_P$ contain $1_P$ for $P \in \Syl_3(I)$, and there is a $3$-element $x \in S$ such that $\alpha^\varsigma(x)=0$ for all $\varsigma \in A$.
\end{lem}

\begin{proof}
We follow the proof of Proposition \ref{typeD-unip}, taking $n=4$, $\eps=+$, and considering the character $\ppsi$ of 
$G = \GO^+_8(q)$, which is trivial at $\ZB(G)$, restricts irreducibly to $S$, and extends to a character $\tilde\ppsi$ of a $\QQ$-representation of
$\Gamma = \langle \CO(V),\phi \rangle$. The only difference is that the character $\alpha:=\ppsi_S$, a unipotent character labeled by the symbol 
$\begin{pmatrix} 2 \\2 \end{pmatrix}$, is not fixed by the triality graph automorphism of $S$, see \cite[Theorem 2.5(ii)]{Ma},
so that $I=I_A(\alpha)$ is $\Gamma/\ZB(\GL(V))$.  
Choosing $\ell = \ppd(q,4)$ (which exists by \cite{Zs}), we see that the character $\pphi_S$ of degree $q(q^2+1)^2$ 
has $\ell$-defect zero, where $\rho=1_G + \pphi+\ppsi$ for the permutation character of $G$ on the set of singular
$1$-dimensional subspaces of $V=\FF_q^8$. As shown in \cite{L2}, one of $\pphi_S$ and $\ppsi_S$ 
contains $1_S$ with multiplicity one in its restriction to $\ell'$-classes, so this must be $\ppsi_S$. 
Applying Lemma \ref{ext}(iii) to the character $\ppsi_S$ of $S \lhd A$, we
see that $\tilde\ppsi$ lies over $1_P$. 

To prove the existence of the desired $3$-element $h \in S$, we again use 
the dual pair $L \times T$ as in the proof of Proposition \ref{typeD-unip}, where $L = \Omega(V) \cong \Omega^+_{8}(q)$, $W=\FF_q^2$,  and 
$T = \Sp(W) \cong \Sp_2(q)$, and  
the identification $\ppsi_S = D_\St -1_S$ where $\St$ is the Steinberg character of $T$.  
Choose $\kappa = \pm 1$ such that $q \equiv \kappa \pmod3$, and $\varep \in \overline{\FF_q}$ of order 
$|\varep|=(q^3-\kappa)_3 > (q \pm 1)_3$, 
and take $\xi=\varep^{(q^3-\kappa)/(q-\kappa)}$ of order $|\xi|=(q-\kappa)_3 > 1$.
Then we choose $h$ to be the image in $S$ of a $3$-element $g \in L$ that belongs to
$\GL^\kappa_3(q) \times \mathsf{C}_{q-\kappa} < \SO^\kappa_{6}(q) \times \SO^\kappa_2(q)$ and has spectrum
$$\bigl\{\varep,\varep^{q},\varep^{q^{2}},\varep^{-1},\varep^{-q},\varep^{-q^{2}},\xi,\xi^{-1}\bigr\}$$
on $V$. We can choose a preimage $\hat{g}$ of $g$ in $\hat{L}=\mathrm{Spin}^+_8(q)$ to have the same order 
$(q^3-\kappa)_3=3^{e+1}$, if $(q-\kappa)_3=3^e$, and view $V$ as an $8$-dimensional representation with kernel of order $\gcd(2,q-1)$. Then any $\varsigma \in \Aut(S)$ sends $g$ to another element of the same order $3^{e+1}$ and the same centralizer order
$(q^3-\kappa)(q-\kappa)$ in $\hat{L}$, and hence its spectrum on $V$ consists of 
$6$ eigenvalues of order $3^{e+1}$, and two of order $3^{e'} >1$, which are inverses to each other. We will now show that 
$\ppsi(g)=0$ for any element $g$ with this kind of spectrum on $V$, which implies that $\alpha^\varsigma(h)=0$ for all $\varsigma \in A$.

Again, $\St(x)=0$ unless $x \in T$ is semisimple. If $x \in T$ is semisimple, then $d(xg)$ is $2$ if 
$x$ is conjugate to $\diag(\xi,\xi^{-1})$, and $0$ otherwise. 
First suppose that $\kappa=+$. Then $(\omega(xg),\St(x))=(q,1)$ for the $q(q+1)$ elements $x \in T$ that are conjugate to $\diag(\xi,\xi^{-1})$, 
$(\omega(xg),\St(x))=(1,q)$ for the two elements $x = \pm I_2$, and $(\omega(xg),\St(x))=(\pm 1, \pm 1)$ for the remaining $q(q^2-3q-2)$
semisimple elements $x \in T$. It follows that $|T|D_\St(g)$ is at most
$$q^2(q+1) +2q+q(q^2-3q-2) =2q^2(q-1) <2|T|$$
and at least
$$q^2(q+1) +2q-q(q^2-3q-2) =4q(q+1)> 0.$$
Since $D_\St(g) \in \ZZ$, we conclude that $D_\St(g)=1$, and so $\ppsi(g)=0$.

Next suppose that $\kappa=-$. Then  
$(\omega(xg),\St(x))=(1,q)$ for the two elements $x= \pm I_2$, 
$(\omega(xg),\St(x))=(1,1)$ for the $q(q+1)(q-3)/2$ elements $x \in T$ that are conjugate to $\diag(b,b^{-1})$
with $b \in \FF_q^\times \smallsetminus \{\pm 1\}$, $(\omega(xg),\St(x))=(\pm q,-1)$ for the $q(q-1)$ elements $x \in T$ that are conjugate to 
$\diag(\xi,\xi^{-1})$, and $(\omega(xg),\St(x))=(\pm 1, \pm 1)$ for the remaining $q(q-1)(q-3)/2$
semisimple elements $x \in T$. It follows that $|T|D_\St(g)$ is at most
$$2q+q(q+1)(q-3)/2 + q^2(q-1)+q(q-1)(q-3)/2 =2q(q-1)^2 <2|T|$$
and at least
$$2q+q(q+1)(q-3)/2 - q^2(q-1)-q(q-1)(q-3)/2 =-q(q-1)^2> -|T|.$$
Since $D_\St(g) \in \ZZ$, we deduce that $D_\St(g) \in \{0,1\}$, whence $\ppsi(g)\in \{-1,0\}$. Suppose now
that $\ppsi(g)=-1$. Then $\rho(g)=1+\pphi(g)+\ppsi(g)=\pphi(g)$. Note that  $\rho(g)$ is the number of $g$-fixed 
singular $1$-dimensional subspaces in $V$, which is zero since no eigenvalue of $g$ belongs to $\FF_q^\times$. 
It follows that $\pphi(g)=0$, which is a contradiction since $g$ is a $3$-element and 
$\pphi(1) = q(q^2+1)^2$ is coprime to $3$. Hence $\ppsi(g)=0$, as stated.
\end{proof}

\begin{prop}\label{slu-odd}
Let $p$ be an odd prime, $q=r^f$ be a power of a prime $r \neq p$, and let $S:=\PSL^\eps_n(q)$ with $\eps=\pm$, 
$p|\gcd(n,q-\eps)$, and $(p,n,q,\eps) \neq (3,3,2,-)$. Then Theorem {\rm\ref{simple1}} holds for $S$.
\end{prop}

\begin{proof}
(a) We first show that if in addition $n > p$ (and so $n \geq 2p \geq 6$), 
then $A=\Aut(S)$ admits a character $\chi$ afforded by a $\QQ$-representation, such that $\chi_S \in \Irr(S)$, 
$\chi_P$ contain $1_P$ for $P \in \Syl_p(A)$, and $\chi(h)=0$ for some $p$-element $h \in S$; in particular, 
Theorem \ref{simple1}(i) holds for $S$.

Assume first that $\eps=+$. Then by Proposition \ref{gl-unip}(ii), $\alpha:=\pc_S$ extends to a character $\tilde\pc$ of a $\QQ$-representation of 
$A$. For any prime divisor $\ell$ of $(q^{n-1}-1)/(q-1)$, using \cite[Proposition 3.1]{GT} we see that $1_G$ is a constituent of multiplicity one 
in the restriction $\pc^\circ$ of $\pc$ to $\ell'$-elements of $G := \GL_n(q)$. Applying Lemma \ref{ext}(iii) to the normal subgroup $G/\ZB(G)$ of
$A$, we deduce that $\chi:=\tilde\pc$ lies over $1_P$. Next, choose $\varep \in \overline{\FF_q}^\times$ of order
$(q^p-1)_p > (q^2-1)_p$. Writing $n=mp$ with $m \in \ZZ_{\geq 2}$, we take $g_+ \in \SL_p(q)$ to be conjugate to
$$\diag\bigl(\varep,\varep^q, \ldots,\varep^{q^{p-1}}\bigr)$$
over $\overline{\FF_q}$, and choose $h$ to be the image in $S$ of a $p$-element 
$g \in \SL_n(q)$, where 
$$g = \diag\bigl(g_+,g_+, \ldots,g_+,g_+^{1-m}\bigr)$$
if $p \nmid (m-1)$, and 
$$g = \diag\bigl(g_+,g_+, \ldots,g_+,g_+^2,g_+^{-m}\bigr)$$
if $p|(m-1)$. This ensures that no eigenvalue of $g$ belongs to $\FF_{q^2}$, and hence $g$ fixes no subspace of dimension $1$ or $2$ of
the natural $\FF_qG$-module $\FF_q^n$. In the notation of Proposition \ref{gl-unip}, this implies that 
$(1_G+\pb)(g)=0=(1_G+\pb+\pc)(g)$, and thus $\pc(g)=0$, as desired. 

Assume now that $\eps=-$. Then by Proposition \ref{gu-unip} and its proof, 
$G:=\GU_n(q)$ has unipotent characters $\pphi$ and $\ppsi$ of degree 
$$\frac{(q^n-(-1)^n)(q^{n-1}+(-1)^nq^2)}{(q+1)(q^2-1)}, ~~\frac{(q^n+(-1)^nq)(q^n-(-1)^nq^2)}{(q+1)(q^2-1)},$$
respectively, both trivial at $\ZB(G)$, irreducible over $S$, and extend to characters $\tilde\pphi$ and $\tilde\ppsi$ of $\QQ$-representations of 
$A$. Choosing $\ell = \ppd(-q,n-1)$ which exists by \cite{Zs}, we see that $\ppsi_S$ has $\ell$-defect $0$. It then follows from \cite{L1} that $1_S$ is a constituent of multiplicity one 
in the restriction $(\pphi_S)^{\circ}$ of $\pphi_S$ to $\ell'$-elements of $S$. Applying Lemma \ref{ext}(iii) to the normal subgroup $S$ of
$A$, we deduce that $\chi:=\tilde\pphi$ lies over $1_P$. Next, choose $\varep \in \overline{\FF_q}^\times$ of order
$(q^p+1)_p > (q^2-1)_p$. Again writing $n=mp$ with $m \in \ZZ_{\geq 2}$, we take $g_- \in \SU_p(q)$ to be conjugate to
$$\diag\bigl(\varep,\varep^{-q}, \ldots,\varep^{(-q)^{p-1}}\bigr)$$
over $\overline{\FF_q}$, and choose $h$ to be the image in $S$ of a $p$-element 
$g \in \SU_n(q)$, where 
$$g = \diag\bigl(g_-,g_-, \ldots,g_-,g_-^{1-m}\bigr)$$
if $p \nmid (m-1)$, and 
$$g = \diag\bigl(g_-,g_-, \ldots,g_-,g_-^2,g_-^{-m}\bigr)$$
if $p|(m-1)$. This ensures that no eigenvalue of $g$ belongs to $\FF_{q^2}$, and so $d(xg)=0$ for any element 
$x \in T:= \GU_2(q)$, in the notation of the proof of Proposition \ref{gu-unip}. We also have $\pphi=D_{1_T}-1_G$, and 
arguing as in part (a1) of the proof of Proposition \ref{gu-unip} we obtain
$D_{1_T}(g)=1$, whence $\pphi(g)=0$, as desired. 

\smallskip
(b) We now show that
Theorem \ref{simple1}(ii) holds for $S$ when $n=p$. We will view $S$ as $[H,H]$, where $H=\GC^F \cong \PGL^\eps_p(q)$ 
for a simple algebraic group $\GC = \PSL_p$ of adjoint type defined over a field of characteristic $r$
and a Steinberg endomorphism $F:\GC \to \GC$. Let the pair $(\GC^*,F^*)$ be dual to $(\GC,F)$ (so that
$\GC^* \cong \SL_p$), and let $L:=(\GC^*)^{F^*} \cong \SL^\eps_p(q)$, whence $|\ZB(L)|=p=|H/S|$. 
Choosing $\ell = \ppd(\eps q,p)$ (which exists by 
\cite{Zs} since we assume $(n,q,\eps) \neq (3,2,-)$), we can find a
semisimple element $s \in L$ of order $\ell$ with centralizer of order $(q^p-\eps^p)/(q-\eps)$, 
and consider the corresponding semisimple character $\chi_s$ of $H$. 
Since $|s|=\ell >p$, $\chi_s$ takes values only in $\QQ(\exp(2\pi i/\ell))$, and hence it is $p$-rational.
Moreover, $\alpha=(\chi_s)_S$ is irreducible by \cite[Proposition 4.3(iii)]{MT1}. Now 
$$\alpha(1)=[L:\CB_L(s)]_{r'} = \frac{|H|_{r'}}{(q^p-\eps^p)/(q-\eps)} = \frac{p|S|_{r'}}{(q^p-\eps^p)/(q-\eps)},$$
whence $\al$ has $p$-defect $0$.

Now let $S \lhd A \leq \Aut(S)$ such that $A/S$ is a $p$-group. As $p > 2$, we may assume that $A \leq H \rtimes \mathsf{C}_f$.
Note $\alpha = (\chi_s)_S$ is certainly $H$-invariant. 
Suppose that some nontrivial field automorphism of
$S$ in the complement $\mathsf{C}_f$, say of order $e>1$ and induced by a bijective morphism
$\sigma$ of $\GC$, fixes $\alpha$. By \cite[Proposition 4.3]{MT1} (see also \cite[Proposition 5.1]{T}), $|s|$ divides $|(\GC^*)^{\sigma^*}|$
for the corresponding bijective morphism $\sigma^*$ of $\GC^*$. But $|(\GC^*)^{\sigma^*}|= |\SL^\eps_p(q^{1/e})|$ is not divisible by 
$\ell$ by the choice of $\ell$, a contradiction. It follows that $I_A(\alpha)=A \cap H$; in particular $\alpha$ extends to the $p$-rational character
$(\chi_s)_{A \cap H}$.

It therefore remains to show that $\chi=\chi_s$ lies over $1_P$ for $P \in \Syl_p(H)$. Note that $Q:= P \cap S \in \Syl_p(S)$ has index 
$p$ in $P$, and $\chi(x)=0$ for $1 \neq x \in Q$. Write $(q-\eps)_p=p^a$ with $a \in \ZZ_{\geq 1}$. 
Suppose first that $p=3$. Using the character table of $\GL^\eps_3(q)$, we can check that $|\chi(x)| \leq 3$ when $x \in P \smallsetminus Q$. Since
$|P|=3^{2a+1}$ and $\chi(1)=(q-\eps)^2(q+\eps) \geq 7 \cdot 3^{2a}$ when $q \geq 7$, it follows that
$$|P|[\chi_P,1_P] = \sum_{x \in P}\chi(x) \geq \chi(1)-\sum_{x \in P \smallsetminus Q}|\chi(x)| \geq 7 \cdot 3^{2a}-3 \cdot 2 \cdot 3^{2a} > 0,$$
i.e. $\chi$ lies over $1_P$. The same can be checked using \cite{GAP} for $S = \PSL_3(4)$ and $\PSU_3(5)$.  

Now we consider the case $p \geq 5$, whence $q \geq 4$. 
We can view $\chi$ as a character of $G=\GL^\eps_p(q)$ which is Lusztig induced from a maximal torus 
$T_1 \cong \mathsf{C}_{q^p-\eps}$. Hence, if $y \in G$ is semisimple and $\chi(y) \neq 0$, it follows from \cite[Proposition 7/5/3]{C} that $y$ is 
conjugate to an element in $T_1$. In particular, $y$ has spectrum $\{\gamma,\gamma^{\eps q}, \ldots,\gamma^{(\eps q)^{p-1}}\}$ on
$\overline{\FF_q}^n$ for 
some $\gamma \in \overline{\FF_q}^\times$ with $\gamma^{q^p-\eps}=1$. Observe that either these $p$ eigenvalues are pairwise distinct, in 
which case $\CB_G(y) \cong T_1$,  or $\gamma^{q-\eps}=1$ and $y \in \ZB(G)$. Assume now that $y$ is a preimage of $x\in P \smallsetminus Q$. Since $\CB_H(S)=1$, we see that $y \notin \ZB(G)$, and so $|\chi(y)| \leq (q^p-\eps)^{1/2} \leq (q^p+1)^{1/2}$ by the preceding observation.  
Now we have $\chi(1) = \prod^{p-1}_{i=1}(q^i-\eps^i) > (1/2)q^{p(p-1)/2}$ and $|P \smallsetminus Q| = (p-1)p^{(p-1)a} \leq q(q+1)^{p-1}$. It follows
that
$$|P|[\chi_P,1_P] \geq \chi(1)-\sum_{x \in P \smallsetminus Q}|\chi(x)| \geq (1/2)q^{p(p-1)/2}-q(q+1)^{p-1}\sqrt{q^p+1} > 0$$
(since $p \geq 5$ and $q \geq 4$),
i.e. $\chi$ lies over $1_P$ in this case as well. 
\end{proof}

\begin{lem}\label{e63}
Let $q=r^f$ be a power of a prime $r \neq 3$, and let $S$ be the simple group of type $E_6^\eps(q)$ with $\eps=\pm$ with
$p=3|(q-\eps)$. Then Theorem {\rm \ref{simple1}(ii)} holds for $S$ when $p=3$.
\end{lem}

\begin{proof}
We will view $S$ as $[H,H]$, where $H=\GC^F=E_6^\eps(q)_{\mathrm{ad}}$ 
for a simple algebraic group $\GC = E_6$ of adjoint type defined over a field of characteristic $r$
and a Steinberg endomorphism $F:\GC \to \GC$. Let the pair $(\GC^*,F^*)$ be dual to $(\GC,F)$, and let 
$L:=(\GC^*)^{F^*} \cong E_6^\eps(q)_{\mathrm{sc}}$, whence $|\ZB(L)|=3=|H/S|$. 
By \cite[Lemma 2.3]{MT} we can find a
semisimple element $s \in L$, of order $\ell=\ppd(q,9)$ with centralizer of order $\Phi_9(q)$ when $\eps=+$, and 
of order $\ell=\ppd(q,18)$ with centralizer of order $\Phi_{18}(q)$ when $\eps=-$. Note that in both cases 
$|\CB_L(s)|_3=3$.
Then the corresponding semisimple character $\chi=\chi_s$ of $H$ takes values only in $\QQ(\exp(2\pi i/\ell))$, and hence it is $3$-rational.
Moreover, $\alpha:=\chi_S$ is irreducible by \cite[Proposition 4.3(iii)]{MT1}, and 
$$\alpha(1)=[L:\CB_L(s)]_{r'} = \frac{|H|_{r'}}{|\CB_L(s)|} = \frac{3|S|_{r'}}{|\CB_H(s)|},$$
whence $\al$ has $3$-defect $0$.

Now let $S \lhd A \leq \Aut(S)$ such that $A/S$ is a $3$-group. Then $A \leq H \rtimes \mathsf{C}_f$.
Note $\alpha = \chi_S$ is certainly $H$-invariant. 
Suppose that some nontrivial field automorphism of
$S$ in the complement $\mathsf{C}_f$, say of order $e>1$ and induced by a bijective morphism
$\sigma$ of $\GC$, fixes $\alpha$. By \cite[Proposition 4.3]{MT1}, $|s|$ divides $|(\GC^*)^{\sigma^*}|$
for the corresponding bijective morphism $\sigma^*$ of $\GC^*$. But $|(\GC^*)^{\sigma^*}|= |E_6^\eps(q^{1/e})_{\mathrm{sc}}|$ is not divisible by 
$\ell$ by the choice of $\ell$, a contradiction. It follows that $I_A(\alpha)=A \cap H$; in particular $\alpha$ extends to the $3$-rational character
$\chi_{A \cap H}$.

It therefore remains to show that $\chi$ lies over $1_P$ for $P \in \Syl_3(H)$. Note that $Q:= P \cap S $ has index 
$3$ in $P$, and $\chi(x)=0$ for $1 \neq x \in Q$. Write $(q-\eps)_3=3^a$ with $a \in \ZZ_{\geq 1}$. 
Suppose first that $S=\tw2E_6(2)$, so $\eps=-$, $|P|=3^{10}$ and $\chi(1)=58615481925$. Using \cite{atlas} we can check that 
$|\chi(x)| \leq 40005$ when $x \in P \smallsetminus Q$. It follows that
$$|P|[\chi_P,1_P] = \sum_{x \in P}\chi(x) \geq  \chi(1)-\sum_{x \in P \smallsetminus Q}|\chi(x)| \geq 58615481925-2 \cdot 3^9 \cdot 40005> 0,$$
i.e. $\chi$ lies over $1_P$. 

Now we may assume $q \geq 4$. Then $|P| \leq 3^4 \cdot (q+1)^6$, and using \cite{Lu} we can check that 
$|\CB_H(x)| < q^{45}(q+1)$ for all $x \in P \smallsetminus Q$, whereas $\chi(1) > (1/2)q^{36}$. It follows
that
$$|P|[\chi_P,1_P] \geq \chi(1)-\sum_{x \in P \smallsetminus Q}|\chi(x)| \geq (1/2)q^{36}-54 \cdot (q+1)^6\sqrt{q^{45}(q+1)} > 0,$$
i.e. $\chi$ lies over $1_P$ in this case as well. 
\end{proof}

Together with Theorem \ref{simple2a}, the following result completes the proof of Theorem \ref{simple1}:

\begin{thm}\label{simple-odd}
Theorem {\rm \ref{simple1}} holds for all finite non-abelian simple groups $S$ when $p>2$.
\end{thm} 

\begin{proof}
By Proposition \ref{almost-simple-others} we need to consider only simple groups $S$ of Lie type in some characteristic $r$.
We may assume that $S$ is not isomorphic to any of the simple groups listed in Lemma \ref{small}.

First suppose that $r=p$. Then the Steinberg character $\St$ of $S$ has $p$-defect zero and of degree $|Q|$ for $Q \in \Syl_p(S)$. It follows
that $[\St_Q,1_Q]=1$. Also, by the main result of \cite{Feit}, $\St$ extends to a rational-valued character of $\Aut(S)$. It follows from Lemma \ref{Lgab}(ii) that Theorem \ref{simple1}(ii) holds for $S$.

We will now assume that $r \neq p$ and view $S$ as $[H,H]$, where $H=\GC^F$ 
for a simple algebraic group $\GC$ of adjoint type defined over a field of characteristic $r$
and a Steinberg endomorphism $F:\GC \to \GC$. Let the pair $(\GC^*,F^*)$ be dual to $(\GC,F)$, and let 
$L:=(\GC^*)^{F^*}$. If $S \cong \PSL^\eps_n(q)$ with $p|\gcd(n,q-\eps)$ or $E_6^\eps(q)$ with $p=3|(q-\eps)$, then 
Theorem \ref{simple1} holds for $S$ by Proposition \ref{slu-odd}, respectively Lemma \ref{e63}. 
So we may further assume that
\begin{equation}\label{for-p1}
  p \mbox{ does not divide }|H/S|=|\ZB(L)|.
\end{equation}

Now let $S \lhd A \leq \Aut(S)$ such that $A/S$ is a $p$-group. Since $p > 2$, Lemma \ref{typeD4-unip} and \eqref{for-p1} allow us to assume that 
$$A \leq \Gamma:=H \rtimes \mathsf{C}_f$$
if $H$ is defined over $\FF_q$ with $q=r^f$, cf. \cite[Theorem 2.5.12]{GLS}. We now show that $H$ admits a $p$-rational character $\chi$ such that 
$\alpha:=\chi_S$ is irreducible of $p$-defect zero and 
\begin{equation}\label{for-p2}
  I_{\Gamma}(\alpha) = H. 
\end{equation}  
It then follows that $I_A(\alpha)=A \cap H$. Since $p \nmid |H/S|$ by \eqref{for-p1}, Theorem \ref{simple1}(ii) holds for $S$.  
We also note that, since $AH/H \cong A/(A \cap H)$ is a $p$-group, one needs to check 
\eqref{for-p2} only when $p|f$, in which case one needs to show that if $\sigma$ is a nontrivial automorphism of $S$ of order $p$ which lies in the complement $\mathsf{C}_f$ to $H$ in $\Gamma$ consisting of field automorphisms, then $\sigma$ does not fix $\alpha$. We may assume that
$\sigma$ is induced by a Steinberg endomorphism, also denoted by $\sigma$, of $\GC$ so that $\sigma^p=F$. In this case, for the
corresponding morphism $\sigma^*$ of $\GC^*$ the subgroup $(\GC^*)^{\sigma^*}$ will be of the same type as of $L$, but defined over 
$\FF_{q^{1/p}}$, and for this reason we denote $(\GC^*)^{\sigma^*}$ by $\GC^*(\FF_{q^{1/p}})$.

\smallskip
(a) Suppose first that $S = \PSL_2(q) \cong \PSU_2(q)$ with $q \geq 7$. Then we view $H$ as $G/\ZB(G)$, where $G = \GU_2(q)$. Consider the case
$p|(q-1)$. Then the character $\chi^{(1,q)}_{q-1}$ of degree $q-1$ 
(in the notation of \cite[p. 28]{E}) is $p$-rational and trivial at $\ZB(G)$, and can be viewed as a
character $\chi$ of $H$ which is also irreducible over $S$ with $p$-defect $0$. 
Using $e \geq p \geq 3$ one can check directly that $\sigma$ does not fix $\alpha$. Thus \eqref{for-p2} holds, and we are done.
Next assume that
$p|(q+1)$. Then the character $\chi^{(q+1)}_{q+1}$ of degree $q+1$ 
(in the notation of \cite[p. 28]{E}) is $p$-rational and trivial at $\ZB(G)$, and can be viewed as a
character $\chi$ of $H$ which is also irreducible over $S$ with $p$-defect $0$. As before, we can check that \eqref{for-p2} holds, and so we are done
in this case as well.

\smallskip
(b) In the remaining cases, we aim to find regular semisimple elements $s_{1,2} \in L$ of order coprime to $|\ZB(L)|$ such that 
\begin{equation}\label{for-p3}
  \gcd(|\CB_L(s_1)|,|\CB_L(s_2)|)_{2'}=|\ZB(L)|_{2'},
\end{equation}
and consider the corresponding semisimple characters $\chi_i=\chi_{s_i}$, $i = 1,2$. By \cite[Proposition 4.3(iii)]{MT1},
$\alpha_i:=(\chi_i)_S$ is irreducible. Next, as $p$ is odd, \eqref{for-p1} and \eqref{for-p3} imply that there exists some $j \in \{1,2\}$ such that
$p \nmid |\CB_L(s_j)|$. Since $\chi_j(1)=[L:\CB_L(s_j)]_{r'}$, it follows that $\chi_j$ and $\alpha_j$ are of $p$-defect zero. Also, 
as $s_j \in \CB_L(s_j)$, $s_j$ is a $p'$-element, which implies that $\chi_j$ is $p$-rational. To ensure that \eqref{for-p2} holds for 
$\alpha:=\alpha_j$, by \cite[Proposition 4.3]{MT1} and the discussion right after \eqref{for-p2}, it suffices to choose $s_i$ so that 
\begin{equation}\label{for-p4}
  \mbox{when }p|f,~s_{1,2}\mbox{ is not }\GC^*\mbox{-conjugate to any element in }(\GC^*)^{\sigma^*}=\GC^*(\FF_{q^{1/p}}).
\end{equation}  
To verify \eqref{for-p4}, we will frequently use the stronger condition that $|s_i| \nmid |\GC^*(\FF_{q^{1/p}})|$ for $i=1,2$.

We will now construct the elements $s_{1,2}$, first in the case $S = \PSL_n(q)$ with $n \geq 3$ and $(n,q) \neq (3,2)$ (note
that $\SL_3(2) \cong \PSL_2(7)$ was already considered). If $(n,q)=(6,2)$, we choose $s_1$ of order $2^6-1$ and $s_2$ of order $2^5-1$, so that 
$|\CB_L(s_1)|=63$ and $|\CB_L(s_2)|=31$, so \eqref{for-p3} and \eqref{for-p4} both hold. 
Suppose $n=3$. Then we choose $s_1$ with $|s_1|=7$ if $(n,q) = (3,4)$, and $|s_1|=\ell_1=\ppd(r,3f) > 3$ otherwise, so $|s_1|$ is coprime 
to $|\ZB(L)|=\gcd(3,q-1)$ and $|\CB_L(s_1)|=(q^3-1)/(q-1)$. 
Next, choose $s_2 \in \SL_2(q) < L$ with $|s_2|=q+1$ which is coprime to $|\ZB(L)|$ and $|\CB_L(s_2)|=q^2-1$.  Then \eqref{for-p3} holds, 
and \eqref{for-p4} holds as well (since when $p|f$, the order of any semisimple element in $\GC^*(\FF_{q^{1/p}})=\SL_3(q^{1/p})$ is at most 
$(q^{3/p}-1)/(q^{1/p}-1) < q+1$ and coprime to $\ell_1$).
In the remaining cases we choose $s_1$ of order $\ppd(r,nf)$, so that $|\CB_L(s_1)|=(q^n-1)/(q-1)$, and $s_2$ with
$|s_2|=\ppd(r,(n-1)f)$ if $(q,n) \neq (4,4)$ and $|s_2|=7$ if $(q,n)=(4,4)$, so that $|\CB_L(s_2)|=q^{n-1}-1$. Then  \eqref{for-p3} holds, 
and \eqref{for-p4} holds as well (since when $p|f$, $|s_i|$ does not divide $|\GC^*(\FF_{q^{1/p}})|=|\SL_n(q^{1/p})|$).

Suppose $S = \PSU_n(q)$ with $n \geq 3$ and $(n,q) \neq (3,2)$.
Suppose first that $n=3$. Then we choose $s_1$ with $|s_1|=\ell_1=\ppd(r,6f) > 3$, so $|s_1|$ is coprime 
to $|\ZB(L)|=\gcd(3,q+1)$ and $|\CB_L(s_1)|=(q^3+1)/(q+1)$. 
Next, choose $s_2 \in \SU_2(q) < L$ with $|s_2|=q-1$ which is coprime to $|\ZB(L)|$ and $|\CB_L(s_2)|=q^2-1$.  Then \eqref{for-p3} holds, 
and \eqref{for-p4} holds as well (since when $p|f$, the order of any semisimple element in $\GC^*(\FF_{q^{1/p}})=\SU_3(q^{1/p})$ is at most 
$q^{2/p}-1 < q-1$ and coprime to $\ell_1$).
In the remaining cases we choose $s_1$ with $|s_1|=\ppd(r,2nf)$ if $2\nmid n$, 
$|s_1|=\ppd(r,nf)$ if $4|n$, and $|s_1|=\ppd(r,nf/2)$ if $n \equiv 2 \pmod4$
so that $|\CB_L(s_1)|=(q^n-(-1)^n)/(q+1)$. We also choose $s_2$ with $|s_2|=\ppd(r,2(n-1)f)$ if $2|n$ and $(n,q) \neq (4,2)$,
$|s_2|=9$ if $(n,q)=(4,2)$,
$|s_2| =\ppd(r,(n-1)f)$ if $4|(n-1)$, $|s_2|=\ppd(r,(n-1)f/2)$ if $n \equiv 3 \pmod4$ and $(n,q) \neq (7,2)$, 
and $|s_2|=63$ if $(n,q)=(7,2)$, so that $|\CB_L(s_2)|=q^{n-1}+(-1)^n$. Then  \eqref{for-p3} holds, 
and \eqref{for-p4} holds as well (again since when $p|f$, $|s_i|$ does not divide $|\GC^*(\FF_{q^{1/p}})|=|\SU_n(q^{1/p})|$).

Assume $S = \Sp_{2n}(q)$ with $2|q$, $n \geq 2$, and $(n,q) \neq (2,2)$ (note that $\Sp_4(2)' \cong \AAA_6$). Then we can choose
$s_1$ with $|s_1|=|\CB_L(s_1)|=q^n-1$ and $s_2$ with $|s_2|=|\CB_L(s_2)|=q^n+1$, which fulfill both \eqref{for-p3} and \eqref{for-p4}.

Next suppose $S=\PSp_4(q)$ with $2 \nmid q$ and $q \geq 9$ (note that $\PSp_4(3) \cong \SU_4(2)$ was already considered). Then 
we take $s_2$ with $|s_2|=(q^2+1)/2$ and $|\CB_L(s_2)|=q^2+1$.
Assume in addition that $q-1$ or $q+1$ is a $2$-power. Then either $q=9$, or $q=r \geq 17$ is a Fermat prime, or $q=r \geq 31$ is a Mersenne prime, whence
\eqref{for-p4} holds vacuously. To fulfill \eqref{for-p3}, we choose
$s_1 \in \SL_2(q) \times \SL_2(q) < L$ with $|s_1|=(q+1)/2$ and $|\CB_L(s_1)|=(q+1)^2$ in the first two cases, and $|s_1|=(q-1)/2$ and $|\CB_L(s_1)|=(q-1)^2$ in the
third case. Now we may assume that neither $q-1$ nor $q+1$ is a $2$-power. Choosing $s_1$ with $|s_1|=(q^2-1)_{2'}$, we have  
$|\CB_L(s_1)|=q^2-1$, which fulfills \eqref{for-p3}. We also note that $|s_2|$ is divisible by $\ppd(r,4f)$, and $|s_1|$ is divisible by $\ppd(r,f)\cdot\ppd(r,2f)$, and hence 
\eqref{for-p4} holds.

Now let $S=\PSp_{2n}(q)$ or $\Omega_{2n+1}(q)$ with $2 \nmid q$, $n \geq 3$, and $(n,q) \neq (4,3)$. Then 
we take $s_2$ with $|s_2|=\ppd(r,2nf)$, so that $|\CB_L(s_2)|=q^n+1$. We also choose $s_1$ with $|s_1|=(q^n-1)_{2'}$ which is divisible by $\ppd(r,nf)$, and also by 
$(q^{n/2}-1)_{2'} \geq 3$ when $2|n$ (as we assume $(n,q) \neq (4,3)$), whence $|\CB_L(s_1)|=q^n-1$. Then one readily checks that both 
\eqref{for-p3} and \eqref{for-p4} hold.

Assume $S=\PO^+_{2n}(q)$ with $n \geq 4$. Then we choose $s_2$ with $|s_2|=\ppd(r,2(n-1)f)$, so that $|\CB_L(s_2)|=(q^{n-1}+1)(q+1)$. 
If $2 \nmid n$, we choose $s_1$ of order $(q^n-1)_{2'}$ which is divisible by $\ppd(r,nf)$, so that $|\CB_L(s_1)|=q^n-1$.
If $2|n$, we choose $s_1$ of order $(q^{n-1}-1)_{2'}$, which is divisible by $\ppd(r,(n-1)f)$ unless $(n,q)=(4,2)$, so that $|\CB_L(s_1)|=(q^{n-1}-1)(q-1)$.
In both cases, one checks that both \eqref{for-p3} and \eqref{for-p4} hold.

Assume $S=\PO^-_{2n}(q)$ with $n \geq 4$. Then we choose $s_2$ with $|s_2|=\ppd(r,2nf)$, so that $|\CB_L(s_2)|=q^{n}+1$. 
If $2 \nmid n$, we choose $s_1$ of order $\ppd(r,2(n-1)f)$, so that $|\CB_L(s_1)|=(q^{n-1}+1)(q-1)$.
If $2|n$, we choose $s_1$ of order $(q^{n-1}-1)_{2'}$, which is divisible by $\ppd(r,(n-1)f)$ unless $(n,q)=(4,2)$, so that $|\CB_L(s_1)|=(q^{n-1}-1)(q+1)$.
In both cases, one checks that both \eqref{for-p3} and \eqref{for-p4} hold.

\smallskip
(c) Finally, we handle the exceptional groups of Lie type. Then we use \cite[Lemma 2.3]{MT} to choose $s_i$ of prime order $\ell_i$, $i=1,2$, as follows.
If $S=\tw2 B_2(q)$ with $q=2^f \geq 8$, then $\ell_1= \ppd(2,4f)$, $\ell_2=\ppd(2,f)$, $|\CB_L(s_1)|= q \pm \sqrt{2q}+1$, and $|\CB_L(s_2)|=q-1$.
If $S=\tw2 G_2(q)$ with $q=3^f \geq 27$ (so $p \geq 5$), then $\ell_1= \ppd(3,6f)$, $\ell_2=\ppd(3,f)$, $|\CB_L(s_1)|= q \pm \sqrt{3q}+1$, and $|\CB_L(s_2)|=q-1$.
If $S=\tw2 F_4(q)$ with $q=2^f \geq 8$, then $\ell_1= \ppd(2,12f)$, $\ell_2=\ppd(2,6f)$, $|\CB_L(s_1)|= q^2+q+1 \pm\sqrt{2q}(q+1)$, and $|\CB_L(s_2)|=q^2-q+1$.
If $S=G_2(q)$ with $q \geq 5$, then $\ell_1= \ppd(r,6f)$, $\ell_2=\ppd(r,4f)$, $|\CB_L(s_1)|= \Phi_6(q)$, and $|\CB_L(s_2)|=\Phi_3(q)$.
If $S=\tw3 D_4(q)$ with $q \geq 3$, then $\ell_1= \ppd(r,12f)$, $\ell_2=\ppd(r,6f)$, $|\CB_L(s_1)|= \Phi_{12}(q)$, and $|\CB_L(s_2)|$ is coprime to $|\CB_L(s_1)|$.
If $S=F_4(q)$, then $\ell_1= \ppd(r,12f)$, $\ell_2=\ppd(r,8f)$, $|\CB_L(s_1)|= \Phi_{12}(q)$, and $|\CB_L(s_2)|=\Phi_8(q)$.
If $S=E_8(q)$, then $\ell_1= \ppd(r,30f)$, $\ell_2=\ppd(r,24f)$, $|\CB_L(s_1)|= \Phi_{30}(q)$, and $|\CB_L(s_2)|=\Phi_{24}(q)$.
If $S=E_6(q)$, then $\ell_1= \ppd(r,9f)$, $\ell_2=\ppd(r,12f)$, $|\CB_L(s_1)|= \Phi_9(q)$, and $|\CB_L(s_2)|=\Phi_{12}(q)\Phi_3(q)$.
If $S=\tw2 E_6(q)$, then $\ell_1= \ppd(r,18f)$, $\ell_2=\ppd(r,12f)$, $|\CB_L(s_1)|= \Phi_{18}(q)$, and $|\CB_L(s_2)|=\Phi_{12}(q)\Phi_6(q)$.
Lastly, suppose $S=E_7(q)$, then we have $\ell_1= \ppd(r,14f)$ and $|\CB_L(s_1)|= q^7+1$ by \cite[Lemma 2.3]{MT}. The proof of \cite[Theorem 4.2]{HSTZ}
yields $s_2$ with $\ell_2= \ppd(r,7f)$ and $|\CB_L(s_2)|= q^7-1$. In all cases, one checks that both \eqref{for-p3} and \eqref{for-p4} hold.
\end{proof}


\section{Proof of Theorem A and of Theorem C}

\begin{lem}\label{lem5}
  Let $N$ be a normal subgroup of a finite group $G$, and let $P\in \mathrm{Syl}_{p}(G)$. 
  Assume that $p\nmid [G:N]$.
  For every $\theta\in \mathrm{Irr}(N)$ and every $\chi\in \mathrm{Irr}(G|\theta)$,
  $\theta$ is an irreducible constituent of $(1_P)^N$ with $p'$-multiplicity
  if and only if $\chi$ is an irreducible constituent of $(1_P)^G$ also with $p'$-multiplicity.
\end{lem}
\begin{proof}
  By Clifford's theorem, $\chi_N=e(\theta_1+\cdots +\theta_t)$, where $\theta_1=\theta$ and $\theta_i=\theta^{x_i}$ for some $x_i\in G$.
  Since $p\nmid [G:N]$, it follows that $p\nmid et$ by Corollary 11.29 of \cite{Is} and all $G$-conjugates of $P$ are contained in $N$.
  In particular, $P^{x_i^{-1}}=P^{y_i}$ for some $y_i\in N$.
  So,
  \[
    [\theta_i,(1_P)^N]= [\theta^{x_i},(1_P)^N]=[\theta,((1_P)^{x_i^{-1}})^N]=[\theta,((1_P)^{y_i})^N]=[\theta,(1_P)^N].         
  \]
  In particular, $[\chi,(1_P)^G]=[\chi_N,(1_P)^N]=[e\sum_{i=1}^t \theta_i,(1_P)^N]=et\cdot [\theta,(1_P)^N]$.
  Consequently, $p\nmid [\theta,(1_P)^N]$ if and only if $p\nmid[\chi,(1_P)^G]$.
\end{proof}

We now proceed to prove Theorem~A and Theorem~C at the same time. 

\begin{proof}[Proof of Theorem~A and of Theorem~C]
  If $P$ is a normal subgroup of $G$, then by Clifford's theorem  every  character $\chi \in \irr G$ that
  lies over $1_P$ contains $P$ in its kernel. As $p$ does not divide the order of the  factor group $G/P$,
  clearly  $(3)$ implies $(1)$.

  \bigskip
  To see that $(1)$ implies $(2)$, as in~\cite{MN} we can observe that every $|P|$-th root of unity
  of the ring  $R$ of algebraic integers is congruent to $1$ mod $pR$. Hence, for $\chi \in \irr G$ and
  $x \in P$, $\chi(x) \equiv \chi(1) \pmod{pR}$. Since $\chi(1) \not\in pR$, it follows that $\chi(x) \neq 0$. 

  \bigskip
  In order to prove that (2) implies (3)
   we assume, working by contradiction, that $G$ is a finite group of  smallest possible order
  such that
  a Sylow $p$-subgroup $P$ of $G$ is not normal in $G$ (in particular, $P \neq 1$), and such that $\chi(x) \neq 0$ for all  $x\in P$ and for all $p$-rational characters $\chi \in
  \irr G$ lying over $1_P$ (for $p=2$: such that $2 \nmid [\chi_P, 1_P]$). (In the following, we will write between brackets the additional conditions that, for $p=2$, we need to consider in order to prove Theorem~C).

  Clearly $G$ is not the trivial group, so it has a minimal normal subgroup $N$.
  Since the $p$-rational irreducible characters of $G/N$ lying over $1_{PN/N}$  lift to $p$-rational irreducible characters of $G$ lying over $1_P$ (and, for $p=2$: multiplicities are preserved), the minimality of $G$ implies that $NP \nor G$ and that $\oh pG = 1$. 

  \medskip
  We first suppose that $N$ is solvable. So $N$ is an elementary abelian group of order coprime to $p$. 
  Let $M/N$ be a chief factor of $G$ with $M \leq NP$. As $\oh pG =1$, the abelian $p$-group $M/N$ acts faithfully on $N$ and hence on the dual group $\irr N$. Therefore, by Brodkey's theorem (\cite[Corollary 3.4]{I2}) there exists a 
  $\lambda \in \irr N$ such that $I_M(\lambda) = N$.  Let $L := \ker{\lambda}$, $I := I_G(\lambda)$ and $T := I \cap NP$. Then $L \nor T$ and $T \nor I$. As $N \leq T$, we have $T = N P_0$
  where $P_0 := T \cap P$ is a Sylow $p$-subgroup of $T$.
   As $T/L = N/L \times LP_0/L$, $\mu := \lambda \times 1_{LP_0/L} \in \irr{T/L}$.
   Let $\widehat{\lambda}$ be the lift of $\mu$ to $T$.
  By~\cite[Corollary 6.27]{Is},  $\widehat{\lambda}$ is the only extension of $\lambda$ to $T$ with determinantal order coprime to $p$. It follows that $\widehat{\lambda}$ is $I$-invariant and hence $\ker{\widehat{\lambda}} = LP_0 \nor I$.
  So, for any  character  $\xi \in \irr{I|\widehat{\lambda}}$, we have  $LP_0 \leq \ker{\xi}$ and hence $\xi$ is $p$-rational since $I/LP_0$ is a $p'$-group.
  Moreover, as $\xi$ lies over $\widehat{\lambda}$, $\xi$ lies both  over  $\lambda$ and $1_{P_0}$.
  (For $p=2$: $\xi_T = \xi(1)\widehat{\lambda}$, so $[\xi_{P_0}, 1_{P_0}] =[(\xi_T)_{P_0}, 1_{P_0}] = \xi(1)[\widehat{\lambda}_{P_0}, 1_{P_0}] =
  \xi(1)$ is odd.)
  By Clifford correspondence $\chi := \xi^G \in \irr G$  and $\chi$ is  $p$-rational as $\Q(\chi) \subseteq \Q(\xi)$.
  Setting $K := INP = IP$, by Mackey's formula  $(\xi^K)_P = (\xi_{P_0})^P$ and hence $\xi^K$ lies over  $1_P$  because $\xi$ lies over $1_{P_0}$. Therefore, $\chi = (\xi^K)^G$ lies over $1_P$.
  (For $p=2$: setting $\psi = \xi^K$,  $  [\psi_P, 1_P] = [(\xi_{P_0})^P, 1_P] = [\xi_{P_0}, 1_{P_0}] = \xi(1)$ is odd.
 Let now $\delta$ be an irreducible constituent of $\psi_{NP}$.
  As $NP \unlhd K$, by Lemma~\ref{lem5} $[\delta_P,1_P]$ is odd. Hence, as $NP \unlhd G$ and $\chi$ lies over $\delta$,
  another application of Lemma~\ref{lem5} yields that $[\chi_P, 1_P]$ is odd.)
  But $\chi$ lies also over $\lambda$ and hence over $\theta = \lambda^M$, because $\lambda^M  \in \irr M$ as $I_M(\lambda) = N$. Hence, by the induction formula $\theta(y) = 0$ for all elements $y \in M$ such that $y \not\in N$. Since $1 \neq|M/N|$ is a $p$-power, there exists  a nontrivial $p$-element $x \in P\cap M$.
 Thus $x^g\not\in N$ and $\theta^{g^{-1}}(x) = \theta(x^g) =0$ for all elements $g \in G$. As $\chi_M$ is a sum
  of $G$-conjugates of $\theta$, we get the contradiction $\chi(x) = 0$.
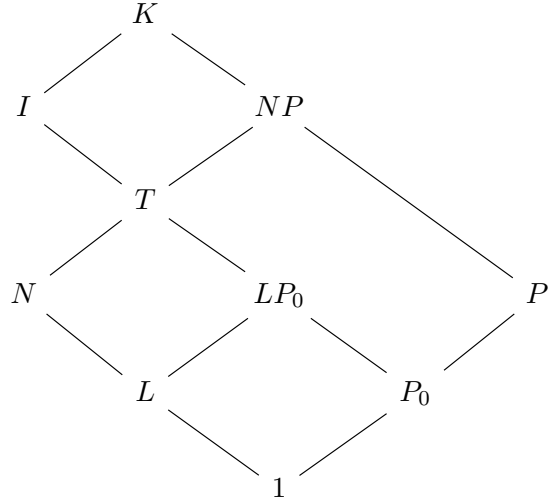
\newline

\begin{figure}[htbp]
  \centering
\begin{tikzcd}[baseline=-0.5ex]
                                          & K \arrow[rd, no head]                                         &                            &                         &   \\
I \arrow[rd, no head] \arrow[ru, no head] &                                                               & NP \arrow[rrdd, no head]   &                         &   \\
                                          & T \arrow[ld, no head] \arrow[rd, no head] \arrow[ru, no head] &                            &                         &   \\
N \arrow[rd, no head]                     &                                                               & LP_0 \arrow[rd, no head] &                         & P \\
                                          & L \arrow[rd, no head] \arrow[ru, no head]                   &                            & P_0 \arrow[ru, no head] &   \\
                                          &                                                               & 1 \arrow[ru, no head]      &                         &  
\end{tikzcd}
\end{figure}
\bigskip
Let us now assume that $N$ is non-solvable. First, we consider the case when $|N|$ is not divisible by the prime $p$ and hence, by Feit--Thompson's theorem, $p \neq 2$.

Write $N = S_1 \times S_2 \times \cdots \times S_n$, where $S_i =  S_1^{g_i}$ for
suitable $g_i \in G$, $1 \leq i \leq n$ (say $g_1=1$), and $S_1$ is a non-abelian simple group of order coprime to $p$.
Let  $W := NP$ and let  

$$K := \bigcap_{i=1}^n \norm W{S_i}$$
be the kernel of the (not necessarily transitive) permutation action of $W$ on the set
$\Omega = \{ S_1, S_2, \ldots, S_n\}$.
So, $ N \leq K \nor W$.
As $W/K$ is a $p$-group and $p \neq 2$, by~\cite[Corollary 5.7(b)]{MW} there exists a subset
$\Delta$ of $\Omega$ such that the setwise stabilizer of $\Delta$
in $W/K$ is trivial.
We can assume that $\Delta = \{S_1, \ldots, S_m\}$, for some $m<n$.
Let now $A = \norm W{S_1}/\cent{W}{S_1}$. Then $A$ is an almost simple group with socle $S$ isomorphic to $S_1$
and $ A/S$ is a $p$-group. 
By~\cite[Lemma 2.6]{MT},  there exist two distinct $A/S$-regular orbits in $\irr S$. Let $\alpha$ and $\beta$ be representatives of such orbits, seen as characters of $S_1$ via the natural isomorphism between $S$ and $S_1$.

We consider the following irreducible character $\theta$ of $N$
$$\theta =  \alpha^{g_1} \times \cdots \times \alpha^{g_m} \times \beta^{g_{m+1}} \times  \cdots \times \beta^{g_n}\,.$$
We claim that $I_W(\theta) = N$. 
In fact, if  $y \in W$ fixes $\theta$, then by considering  $W$ as a subgroup of the wreath product $A \wr (W/K)$ 
one sees that $y$ stabilizes the set $\Delta$, as the  characters $\alpha$ and $\beta$ lie in different $A$-orbits.
So, we conclude that $y \in K$.
Then, since $I_{\norm W{S_i}}(\alpha^{g_i})=N\cent W{S_i}$ for $1\leq i\leq m$
and $I_{\norm W{S_i}}(\beta^{g_i}) = N\cent W{S_i}$ for $m+1 \leq i \leq n$, we deduce that
    $$y \in  \bigcap_{i= 1}^n N\cent W{S_i} = N.$$
    Thus, $I_W(\theta) = N$ and hence $\tau := \theta^W$ is a $p$-defect zero irreducible  character of $W$.
    By Lemma~\ref{Lprat} there exists a $p$-rational character $\chi \in \irr G$ such that $\chi$ lies over $\tau$.
    As $W \nor G$, by Clifford theorem $\chi_W$ is a sum of conjugate characters $\tau^g\in \irr W$, for suitable elements $g \in G$. 
    Since by Lemma~\ref{Ldef0} $\tau$ lies over $1_P$ and $\tau(x) = 0$ for every nontrivial element $x \in P$,
    $\chi$ also lies over $1_P$ and $\chi(x) = 0$ for every nontrivial element $x \in P$.
    As $P \neq 1$, we have a contradiction.  
    
    \bigskip
    Let us finally assume that $N$ is non-solvable and that $p$ divides $|N|$.
    Again, we write  $N = S_1 \times S_2 \times \cdots \times S_n$,
    where $S := S_1$ is a non-abelian simple group such that $p$ divides $|S|$ and $S_i=S^{g_i}$ for some $g_i\in G$, $i = 1, 2, \ldots, n$.
  Let $D=\mathbf{N}_{G}(S)$,  $C=\mathbf{C}_{G}(S)$ and $E = SC = S \times C$.
  Then $D/C$ is an almost simple group with socle $E/C$, isomorphic to $S$.
  Since $N \leq E$ and $G/N$ has a normal Sylow $p$-subgroup, then $D/E$  has a normal Sylow $p$-subgroup as well.

  We now state the following

  \medskip
  ({\bf Claim}):  there exist a character $\alpha \in \mathrm{Irr}(E/C)$ and  $xC\in E/C$, with
  $x$ a $p$-element of  $S$, 
  such that $\alpha^{\varsigma}(xC)=0$ for all $\varsigma \in \mathrm{Aut}(E/C)$, and a $p$-rational character
  $\beta \in \mathrm{Irr}(J/C)$ lying above both $\alpha$ and  $1_{QC/C}$, where $J/C:=I_{D/C}(\alpha)$ and $Q\in \mathrm{Syl}_{p}(J)$ (for $p=2$: such that $[\beta_{QC/C}, 1_{QC/C}]$ is odd).

  \medskip

  In order to prove the claim, we set $\overline{D} := D/C$ and we use the bar convention (so, $\overline{E} \cong S$). 
  We start by assuming  that $\overline{E}$ satisfies part (ii) of Theorem~\ref{simple1}. So, there exists a $p$-defect zero character
  $\alpha \in \mathrm{Irr}(\overline{E})$ such that, for $\overline{J}:=I_{\overline{D}}(\alpha)$,
  $\overline{T}/\overline{E}$ the normal Sylow $p$-subgroup of $\overline{J}/\overline{E}$ and
  $\overline{Q}\in \mathrm{Syl}_{p}(\overline{J})$,
 there exists a $p$-rational character $\gamma \in \mathrm{Irr}(\overline{T})$ such that $\gamma$ extends $\alpha$ and $\gamma$
  lies over $1_{\overline{Q}}$. (Moreover, for $p=2$:  $[\gamma_{\overline{Q}}, 1_{\overline{Q}}]$ is odd).
  By Schur-Zassenhaus' theorem there exists a complement  $\overline{H}/\overline{E}$ of $\overline{T}/\overline{E}$ in  $\overline{J}/\overline{E}$ and  by Lemma~\ref{Lprat}
  there exists a $p$-rational  character $\eta \in \irr{\overline{H}}$ that lies over $\alpha$.
 Now, by~\cite[Lemma 6.8]{N} the  restriction of characters of $\overline{J}$ to $\overline{H}$ defines a bijection from the set of the irreducible characters of $\overline{J}$ lying over $\gamma$ to the set of the irreducible characters of $\overline{H}$ lying over $\alpha$. 
Let now $\beta\in \irr{\overline{J}}$ be the unique character lying over $\gamma$ such that $\beta_{\overline{H}} = \eta$.
Thus, if $\sigma$ is a Galois automorphism that fixes both $\gamma$ and $\eta$, then 
$\sigma$ fixes also $\beta$. Since both $\gamma$ and $\eta$ are $p$-rational, we conclude that
$\beta$ is $p$-rational as well.
Moreover, $\beta$ lies over $\gamma$ and $\gamma$ lies over $1_{\overline{Q}}$, so $\beta$ lies over $1_{\overline{Q}}$.
(For $p=2$: $\overline{T}$ is a normal subgroup of odd index of $\overline{J}$ and $[\gamma_{\overline{Q}}, 1_{\overline{Q}}]$ is odd,
so by Lemma~\ref{lem5} $[\beta_{\overline{Q}}, 1_{\overline{Q}}]$ is odd).
Finally, let $x \in S$ be any nontrivial $p$-element; such an element exists as $p$ divides $|S|$.
For all $\varsigma \in \mathrm{Aut}(\overline{E})$,
$\alpha^{\varsigma}$ is a  $p$-defect zero character of $\overline{E}$ and hence  $\alpha^{\varsigma}(\overline{x}) =0$. 

If $\overline{E}$ satisfies part (i) of Theorem~\ref{simple1} and $p \neq 2$, then the Claim clearly holds. 
We can now assume that $p=2$ and we observe that if there exists a character $\alpha \in \mathrm{Irr}(\overline{E})$
of $2$-defect zero, then by Lemma~\ref{Lgab}(iii) $\overline{E}$ satisfies part (ii) of Theorem~\ref{simple1} and hence we are done by the preceding paragraph.
On the other hand, the Claim clearly holds if $\overline{E} = \mathrm{Aut}(\overline{E})$. 
Then by \cite[Corollary 2]{GO} and \cite{atlas} we can assume that $\overline{E}$ is either an alternating group of degree at least $7$ or
$\overline{E} \in\mathcal{L} = \{M_{12}, M_{22}, J_2, HS, Suz \}$.
In particular,   $[\overline{D}:\overline{E}] \leq 2$. If $\overline{E}$ is an alternating group of degree at least $7$, then we are done by \cite[Theorem 3.17]{GLLV}.
Finally, if $\overline{E}\in \mathcal{L}$, then using \cite{GAP} one verifies, as in Lemma~\ref{small}, that
there exist a rational character  $\alpha \in \mathrm{Irr}(\overline{E})$ and an element $\overline{x} \in \overline{P}$ with $x\in S$ a $p$-element, where $\overline{P} \in \syl p{\overline{E}}$, such that $[\alpha_{\overline{P}}, 1_{\overline{P}}]$ is odd and $\alpha(\overline{y}) = 0$ for every $\overline{y} \in \overline{E}$ with $|\overline{y}| = |\overline{x}|$. Hence, another application of Lemma~\ref{Lgab}(iii) finishes the proof of the Claim. 

\medskip

 In the following, we  adopt the notation of the  Claim. We first observe that $\beta$ lies over $1_Q$ and that, without loss of generality, we may
  assume $Q\leq P$.

  By seeing $\alpha$ as a character of $E = S \times C$, we write $\alpha= \alpha_1 \times 1_C$, with $\alpha_1\in\mathrm{Irr}(S)$,  and
  we define  $$\theta= \alpha_N = \alpha_1 \times 1_{S_2}\times \cdots \times 1_{S_n}$$
  and $I = I_G(\theta)$.
  We observe that every element $g \in I$ fixes also $\mathrm{ker}(\theta) = S_2 \times \cdots \times S_n$ and hence
  $g\in D$. 
  It follows that  $I = I_D(\theta) = I_{D}(\alpha_1)=I_D(\alpha)=J$.
  We see now  $\beta$ as an irreducible character of $I = J$ and observe that $\beta$ lies over $\theta$.
  Thus, Clifford's correspondence  yields that
  $\chi:=\beta^G\in \mathrm{Irr}(G)$.
  Moreover, since $\beta$ is $p$-rational, $\chi$ is also $p$-rational.
  
  We now set $X=I NP$.
  Observe that $N\leq I$, and so $X=IP$.
  As $Q \leq P$, then $Q = I \cap P$ and by~Mackey's formula  $(1_Q)^I = ((1_P)_{P\cap I})^I = ((1_P)^X)_I$ and hence 
  \[
  [\beta^{X},(1_P)^{X}] = [\beta,((1_P)^{X})_I] = [\beta,(1_Q)^{I}]= [\beta_Q, 1_Q] =[\beta_{QC/C}, 1_{QC/C}]  
\]
is  nonzero. 
  We conclude that  $\chi=(\beta^X)^G$ lies above $1_P$.
  (For $p=2$: setting $\psi = \beta^X$, again by Mackey's formula  $[\psi_P, 1_P]=[(\beta_Q)^{P},1_P]= [\beta_Q, 1_Q] =  [\beta_{QC/C}, 1_{QC/C}]$ is odd. Let now $\mu$ be an irreducible constituent of $\psi_{NP}$.
  As $NP \unlhd X$, by Lemma~\ref{lem5} $[\mu_P,1_P]$ is odd. Hence, as $NP \unlhd G$ and $\chi$ lies over $\mu$,
  another application of Lemma~\ref{lem5} yields that   $[\chi_P, 1_P]$ is odd.) 
  
  However, $\chi$ lies also over $\theta$ and hence by Clifford's theorem
  $\chi_N$ decomposes as a sum of $G$-conjugate characters $\theta^g$, with $g \in G$.
  For $g \in G$, there exist  $\varsigma \in \mathrm{Aut}(S)$ and $i \in \{1,\ldots, n\}$, such that
  $$\theta^g = \theta^{\varsigma g_i} = \alpha_1^{\varsigma g_i} \times 1_{N_i}$$
  where $N_i = \prod_{j\neq i}S_j$. 
  As $\alpha_1^{\varsigma g_i}(x^{g_i}) = \alpha_1^{\varsigma}(x) =\alpha^{\varsigma}(xC) = 0$ for every $\varsigma \in\mathrm{Aut}(S) = \mathrm{Aut}(E/C)$,  
we deduce that $\chi$  vanishes on the $p$-element $\prod_{i=1}^t x^{g_i}$, the final contradiction.
\end{proof}

\section{Proof of Theorem B }

Before proving Theorem B, we show a direct consequence of Theorem \ref{simple2}.

\begin{prop}\label{cor, simple2}
Let $A$ be an almost simple group with socle $S$.
Assume that $A/S$ has a normal Sylow $2$-subgroup.
Then there exist  an irreducible character $\alpha$ of $S$ and a $2$-element $x \in S$ such that 
$\alpha^\varsigma(x)=0$ for all $\varsigma \in \Aut(S)$ and  
a strongly real irreducible character $\beta$ of $I=I_{A}(\alpha)$ such that $\beta$  lies over both $\alpha$ and  $1_P$, where $P$ is a Sylow $2$-subgroup of $I$. 
\end{prop}
\begin{proof}
  Assume first that part (i) of  Theorem \ref{simple2} holds for $S$. 
  Then $S$ admits an irreducible character $\alpha$ and a $2$-element $x \in S$ such that $\alpha^\varsigma(x) = 0$ for all $\varsigma \in \Aut(S)$, and $\alpha$ extends to a strongly real character $\gamma$ of $T := I_{\Aut(S)}(\alpha)$ lying above $1_Q$, where $Q \in \mathrm{Syl}_2(T)$. Let $I := T \cap A = I_A(\alpha)$. 
  Without loss of generality, we may choose $Q$ such that $P := I \cap Q$ is a Sylow $2$-subgroup of $I$. 
  Consequently, the restriction $\beta := \gamma_I \in \mathrm{Irr}(I)$ is a strongly real character lying above both $\alpha$ and $1_P$.

  Assume next that part (ii) of Theorem \ref{simple2} holds for $S$. 
  Let $H/S \in \mathrm{Syl}_2(A/S)$ and observe that $H \unlhd A$. 
  Then there exists a character $\alpha \in \mathrm{Irr}(S)$ of $2$-defect zero such that $\alpha$ extends to a strongly real character $\gamma$ of $I := I_H(\alpha)$ lying above $1_P$, where $P \in \mathrm{Syl}_2(I)$. 
  Set $T = I_A(\alpha)$. 
  Since $[T : I]$ is odd, $P$ is also a Sylow $2$-subgroup of $T$ and, 
  as $\gamma$ lies over $\alpha$ and $1_P$, every irreducible character of $T$ lying over $\gamma$ must also lie over $\alpha$ and $1_P$. 
  Moreover, by~\cite[Lemma 2.1(ii)]{MT1} there exists a strongly real character $\beta \in \mathrm{Irr}(T)$ such that $\beta$ lies over $\gamma$,
  and consequently over both $\alpha$ and $1_P$.
\end{proof}

Now, we are ready to prove Theorem B.

\begin{proof}[Proof of Theorem B]
  As in  the proof of Theorem~A and Theorem~C, it follows easily that  (3) implies (1) and that (1) implies (2).
  
We show  that (2) implies (3).
Working by induction on $|G|$, we can assume that $G$ has a unique minimal normal subgroup  $N$,
$N$ is not a $2$-group, and that $NP$ is a normal subgroup of $G$, $N < PN$.

We assume first that $|N|$ is odd. 
Hence,  $N$ is an elementary abelian $q$-group for some odd prime $q$. 
Let $K/N$ be a chief factor of $G$,  with  $K \leq NP$. So, $K/N$ is an elementary abelian $2$-group and
$\mathbf{O}_2(K) \leq \mathbf{O}_2(G) = 1$.
By Corollary~3.4 of~\cite{I2} and Theorem~2.5 of~\cite{MT1}, there exists a character $\nu \in \irr N$ and
an element $y\in K$ such that $\nu^y = \overline{\nu}$ and  $\tau := \nu^K \in \mathrm{Irr}(K)$.
We observe that $\tau$ is a character of $2$-defect zero of $K$, and that $\overline{\tau} = (\overline{\nu})^K = (\nu^y)^K = \nu^K = \tau$. 
Hence, by Lemma~\ref{triv}(i) $\tau$ is strongly real and by Lemma~\ref{Ldef0}  $[\tau_Q, 1_Q]$ is odd, where
$Q$ is a Sylow $2$-subgroup of $K$.

Now, let $M := NP = KP$,  $I = I_{M}(\tau)$ and $Q_1=I\cap P$.
We remark that  $Q_1$ is a Sylow $2$-subgroup of $I$,  because $IP=M$ and hence $[I:Q_1]=[M:P]$. 
Applying Lemma \ref{Lgab}(iii) to $(I,K,\tau)$, we deduce that $\tau$ extends to a strongly real character $\phi \in \mathrm{Irr}(I)$ such that $[\phi, (1_{Q_1})^I]$ is odd. Consequently, the induced character $\psi := \phi^{M} \in \mathrm{Irr}(M)$ is also strongly real. 
Moreover,$$[\psi, (1_P)^{M}] = [\phi^{M}, (1_P)^{M}] = [\phi, ((1_P)^{M})_I] = [\phi, (1_{Q_1})^I]$$is odd. 
In particular, $\psi$ lies above $1_P$.
Since $[G : M]$ is odd, \cite[Lemma 2.1(ii)]{MT1} implies that there exists a (unique) strongly real character $\chi \in \mathrm{Irr}(G)$ lying above $\psi$. 
So $\chi$ also lies above $1_P$. 
Finally, by Clifford's theorem $\chi_K$ is a sum of $G$-conjugates of $\tau$. Because $\tau$ has $2$-defect zero in $K$, it follows that $\chi$ vanishes on every nontrivial $2$-element in $K$, which contradicts (2).

\medskip
  Assume now that $|N|$ is even.
  Then (as $G$ has no non-trivial normal subgroups of $2$-power order)
  $N=S_1\times \cdots \times S_t$, where $S := S_1$ is a non-abelian simple group and $S_i=S^{g_i}$ for some $g_i\in G$, $i = 1, 2, \ldots, t$.
  Let $D=\mathbf{N}_{G}(S)$,  $C=\mathbf{C}_{G}(S)$ and $E = SC = S \times C$.
  Then $D/C$ is an almost simple group with socle $E/C$, isomorphic to $S$.
  Since $N \leq E$ and $G/N$ has a normal Sylow $2$-subgroup, then $D/E$  has a normal Sylow $2$-subgroup as well.
  Applying Proposition~\ref{cor, simple2} to $D/C$, 
  we deduce that there exist a character $\alpha \in \mathrm{Irr}(E/C)$ and an element  $xC\in E/C$, with
  $x$ a $2$-element of  $S$, 
  such that $\alpha^{\varsigma}(xC)=0$ for all $\varsigma \in \mathrm{Aut}(E/C)$, and a strongly real character
  $\beta \in \mathrm{Irr}(J/C)$ lying above both $\alpha$ and  $1_{QC/C}$, where $J:=I_D(\alpha)$ and $Q\in \mathrm{Syl}_{2}(J)$.
  We observe that, in particular, $\beta$ lies over $1_Q$ and that, without loss of generality, we may
  assume $Q\leq P$.
  By lifting, we write $\alpha= \alpha_1 \times 1_C$, with $\alpha_1\in\mathrm{Irr}(S)$,  and
  we define  $$\theta= \alpha_N = \alpha_1 \times 1_{S_2}\times \cdots \times 1_{S_t}$$
  and $I = I_G(\theta)$.
  We observe that every element $g \in I$ fixes also $\mathrm{ker}(\theta) = S_2 \times \cdots \times S_t$ and hence
  $g\in D$. 
  It follows that  $I = I_D(\theta) = I_{D}(\alpha_1)=I_D(\alpha)=J$.
  By seeing $\beta$ as an irreducible character of $I$, Clifford's correspondence then  yields that
  $\chi:=\beta^G\in \mathrm{Irr}(G)$.
  Moreover, since $\beta$ is strongly real, $\chi$ is also strongly real.
  
  We now set $X=I NP$.
  Observe that $N\leq I$, and so $X=IP$.
  As $Q = I \cap P$, by~Mackey's formula  $(1_Q)^I = ((1_P)_Q)^I = ((1_P)^X)_I$ and hence 
  \[
   [\beta^{X},(1_P)^{X}]=[\beta,((1_P)^{X})_I]=[\beta,(1_Q)^{I}]\neq 0.
  \] 
  We conclude that $\beta^X$, and hence $\chi=\beta^G$ lies above $1_P$.
  However, $\chi$ lies also over $\theta$ and hence by Clifford's theorem
  $\chi_N$ decomposes as a sum of $G$-conjugate characters $\theta^g$, for suitable $g \in G$.
  For any given $g \in G$, there exist a $\varsigma \in \mathrm{Aut}(S)$ and an  $i \in \{1,\ldots, t\}$, such that
  $$\theta^g = \theta^{\varsigma g_i} = \alpha_1^{\varsigma g_i} \times 1_{N_i}$$
  where $N_i = \prod_{j\neq i} S_j$. 
  As $\alpha_1^{\varsigma g_i}(x^{g_i}) = \alpha_1^{\varsigma}(x) = 0$ for every $\varsigma \in\mathrm{Aut}(S)$,  
we deduce that $\chi$  vanishes on the $2$-element $\prod_{i=1}^t x^{g_i}$, a contradiction.
 \end{proof}

\end{document}